\documentclass[12pt]{amsart}
\usepackage{amsmath,mathtools,amsthm,amsfonts,bigints,amssymb, mathrsfs, tikz-cd, xcolor, float}
\usepackage{tikz}
\usetikzlibrary{patterns}
\usepackage{pgf,tikz}

\usepackage{pgfplots}
\usepgfplotslibrary{fillbetween}

\usetikzlibrary{arrows}

\newtheorem{theorem}{Theorem}[section]
\newtheorem{lemma}[theorem]{Lemma}

\theoremstyle{definition}
\newtheorem{definition}[theorem]{Definition}

\newtheorem{remark}[theorem]{Remark}

\newcommand{\what}{\widehat}

\newcommand{\R}{\mathbb R}%
\newcommand{\C}{\mathbb C}%
\newcommand{\Z}{\mathbb Z}%
\newcommand{\N}{\mathbb N}%
\newcommand{\Sb}{\mathbb S}%

\newcommand{\Ac}{\mathcal A}%
\newcommand{\T}{\mathbb T}%
\newcommand{\X}{\mathbb X}%
\newcommand{\Hb}{\mathbb H}%
\newcommand{\sa}{\textsf{a}}
\numberwithin{equation}{section}

\usepackage[colorlinks=true,linkcolor=blue,citecolor=blue]{hyperref}
\makeatletter

\renewcommand\subsubsection{\@secnumfont}{\bfseries}%
\renewcommand\subsubsection{\@startsection{subsubsection}{3}
  \z@{.5\linespacing\@plus.7\linespacing}{-.5em}%
  {\normalfont\bfseries}}

  \makeatother
\makeatletter
\@namedef{subjclassname@2020}{%
  \textup{2020} Mathematics Subject Classification}
\makeatother

\begin{document}

\title[Regularity and pointwise convergence]{Regularity and pointwise convergence for dispersive equations on Riemannian symmetric spaces of compact type }

\author[U. Dewan]{Utsav Dewan}
\address{Department of Mathematics, Indian Institute of Technology Bombay, Powai, Mumbai-400076, India}
\email{utsav@math.iitb.ac.in}

\author[S. Pusti]{Sanjoy Pusti}
\address{Department of Mathematics, Indian Institute of Technology Bombay, Powai, Mumbai-400076, India}
\email{sanjoy@math.iitb.ac.in}



\begin{abstract}
In this article, we first prove that for general dispersive equations on Riemannian symmetric spaces of compact type $\mathbb{X}=U/K$, of rank $1$ and $2$, the Sobolev regularity thresholds for the initial data, $\alpha >1/2$ and $\alpha >1$ respectively, are sufficient to obtain pointwise convergence of the solution a.e. on $\mathbb{X}$. We next focus on $K$-biinvariant initial data for rank $1$ and prove that the sufficiency of the regularity threshold can be improved down to $\alpha>1/3$, whereas the phenomenon fails for $\alpha<1/4$ for the Schr\"odinger equation. We also obtain the same results for other dispersive equations: the Boussinesq equation and the Beam equation, also known as the fourth order Wave equation, by a novel transference principle, which seems to be new even for the circle $\mathbb{T} \cong SO(2)$ and may be of independent interest. Our arguments involve harmonic analysis arising from the representation theory of compact semi-simple Lie groups and also number theory.
\end{abstract}

\maketitle
\tableofcontents

\section{Introduction}
One of the most celebrated problems in Euclidean Harmonic analysis is the Carleson's problem: determining the optimal regularity of the initial condition $f$ of the
Schr\"odinger equation given by
\begin{equation} \label{schrodinger}
\begin{cases}
	 i\frac{\partial u}{\partial t} +\Delta_{\R^d} u=0\:,\:\:\:  &(x,t) \in \R^d \times \R\:, \\
	u(\cdot, 0)=f\:, &\text{ on } \R^d \:,
	\end{cases}
\end{equation}
in terms of the index $\alpha$ such that for all $f$ belonging to the inhomogeneous Sobolev space $H^\alpha(\R^d)$, the solution $u$ of (\ref{schrodinger}) converges pointwise to $f$ a.e., that is, 
\begin{equation} \label{pointwise_convergence}
\displaystyle\lim_{t \to 0+} u(x,t)=f(x)\:,\:\:\text{ almost everywhere }.
\end{equation}
In 1980, Carleson studied this problem for $d=1$ in \cite{C} and obtained the sufficient condition $\alpha \ge 1/4$ for the pointwise convergence (\ref{pointwise_convergence}) to hold. In 1982, Dahlberg-Kenig \cite{DK} showed that $\alpha \ge 1/4$ is also necessary. This completely solved the problem in dimension one and subsequently, posed the question in higher dimensions.

\medskip

In 1983, Cowling \cite{Cowling} studied the problem for a general class of self-adjoint operators on $L^2(X)$ for a measure space $X$  and obtained $\alpha > 1$, to be a sufficient condition for the associated Schr\"odinger operator. For $\R^d$, in 1987, Sj\"olin \cite{Sjolin} improved the bound down to $\alpha > 1/2$, by means of a local smoothing effect. This improvement was also independently obtained by Vega \cite{Vega} in 1988. Then developing decoupling techniques, continuous improvements were made by several mathematicians \cite{Bourgain1, MVV, Tao, Lee}. These attempts reached fruition recently in 2017, when Du-Guth-Li \cite{DGL} obtained the bound $\alpha > 1/3$, for $\R^2$ by using polynomial partitioning and decoupling. Then in 2019, Du-Zhang \cite{DZ} obtained the bound $\alpha > d/2(d+1)$, for dimensions $d \ge 3$, by using decoupling and induction on scales. This bound is sharp except at the endpoint $\alpha =d/2(d+1)$, due to a counterexample by Bourgain \cite{Bourgain}. Thus in the setting of $\R^d$, the Carleson's problem has been almost fully resolved.

\medskip

In the case of periodic initial data, that is on $\T^d$, much less is known however. In dimension one, employing Strichartz estimates, Moyua-Vega showed in \cite{MV} that the bound $\alpha > 1/3$ is sufficient, whereas for $\alpha <1/4$ the pointwise convergence fails. For $d \ge 2$, extending the above ideas along with the improved Strichartz estimates obtained in \cite{BD}, Compaan-Luc\`a-Staffilani showed in \cite{CLS} that the bound $\alpha>d/(d+2)$ is sufficient, whereas for $\alpha <d/2(d+1)$ the pointwise convergence fails.

\medskip 

For compact non-Euclidean Riemannian manifolds even less is known. In \cite{WZ}, Wang-Zhang showed that on the unit sphere $\Sb^d$, the bound $\alpha>1/2$ is sufficient. Moreover, for lower dimensional connected, compact Riemannian manifolds without boundary, they proved the sufficiency of the bound
\begin{equation} \label{general_manifold}
\alpha > 
\begin{cases}
	 \frac{3}{4} &\text{ if } d=2 \\
	\frac{9}{10} &\text{ if } d=3\:.
	\end{cases}
\end{equation}
In addition, if all geodesics are closed with a common period, that is, on a Zoll manifold, they could improve (\ref{general_manifold}) down to
\begin{equation} \label{zoll_manifold}
\alpha > 
\begin{cases}
	 \frac{5}{8} &\text{ if } d=2 \\
	\frac{3}{4} &\text{ if } d=3\:.
	\end{cases}
\end{equation}

The most well-known examples of Zoll manifolds are given by the class of rank one Riemannian symmetric spaces of compact type. Riemannian symmetric spaces of compact type are homogeneous spaces $\X=U/K$, where $U$ is a connected, simply connected, compact, semi-simple Lie group and $K$ is a closed subgroup with the property that $U^\theta_0 \subset K \subset U^\theta$ for an involution $\theta$ of $U$. Here $U^\theta$ denotes the subgroup of $\theta$-fixed points and $U^\theta_0$ its identity component. In the special case when rank$(\X)=1$, one has a well-known classification of such spaces \cite{Wang}: 
\begin{itemize}
\item the sphere $\Sb^d =SO(d+1)/SO(d)$, $d=1,2,3,\dots$;
\item the real projective space $P^d(\R)=SO(d+1)/O(d)$, $d=2,3,4,\dots$;
\item the complex projective space $P^d(\C)=SU(l+1)/S(U(l) \times U(1))$, $d=4,6,8,\dots$ and $l=d/2$;
\item the quarternionic projective space $P^d(\Hb)=Sp(l+1)/Sp(l) \times Sp(1)$, $d=8,12,16,\dots$ and $l=d/4$;
\item the Cayley projective plane $P^{16}(Cay)$.
\end{itemize}
Here $d$ denotes the real dimension of any of these spaces, $O(d),\:U(d),\:Sp(d)$ denote the orthogonal, unitary and symplectic groups of order $d$ and $S(\cdot)$ denotes the formation of a subgroup of matrices of unit determinant. The Riemannian volume measure on $\X$ is induced by the normalized Haar measure on $U$. For the notion of the rank of a symmetric space and other unexplained terminologies, we refer the reader to Section $2$.
 
\medskip
 
In this article, in the setting of Riemannian symmetric spaces of compact type $\X$, we will study the pointwise convergence of solutions of dispersive equations 
\begin{equation} \label{dispersive}
\begin{cases}
	 i\frac{\partial u}{\partial t} -\psi\left(\sqrt{-\Delta}\right) u=0\:,\:\:\:  &(x,t) \in \X \times [0,2\pi)\:, \\
	u(\cdot, 0)=f\:, &\text{ on } \X \:,
	\end{cases}
\end{equation}
where $\psi: [0,\infty) \to \R$ satisfies suitable conditions and $\Delta$ is the Laplace-Beltrami operator on $\X$ with non-positive, discrete $L^2$-spectrum. We note that for $\psi(r)=r^a$ with $a>1$, we recover the fractional Schr\"odinger equations with convex phase and moreover, $a=2$ corresponds to the Schr\"odinger equation. Some other examples are given by $\psi(r):= r \sqrt{1+r^2}$ and $\psi(r):= \sqrt{1+r^4}$ corresponding respectively to the Boussinesq equation and the Beam equation, also known as the fourth order Wave equation (for more details see \cite{FG,GPW}), which often appear in Mathematical Physics.

\medskip

To quantify the regularity of the initial data, we consider the inhomogeneous Sobolev spaces:
\begin{equation} \label{Sobolev_space}
H^\alpha\left(\X\right):=\left\{f \in L^2 \left(\X\right) \mid \left(I-\Delta\right)^{\frac{\alpha}{2}}f \in L^2 \left(\X\right) \right\}\:,\:\alpha \ge 0\:.
\end{equation}
Our first result yields a sufficiency bound on the regularity of the initial data for general dispersive equations:
\begin{theorem} \label{result1}
Let $\X$ be a Riemannian symmetric space of compact type of rank $1$ or $2$. For dispersive equations (\ref{dispersive}), the pointwise convergence of the solution to its initial data,
\begin{equation*}
\displaystyle\lim_{t \to 0+} u(x,t)=f(x)\:,
\end{equation*}
holds for almost every $x \in \X$, whenever $f \in H^\alpha(\X)$ with 
\begin{equation*} 
\alpha > \begin{cases}
	 \frac{1}{2}  &\text{ if } rank(\X)=1\:, \\
	 1  &\text{ if } rank(\X)=2\:.
	\end{cases}
\end{equation*}
\end{theorem} 

Some remarks are now in order :
\begin{remark} \label{remark_result1}
In the setting of Riemannian symmetric spaces of compact type of rank one,
\begin{enumerate}
\item for the Schr\"odinger equation,  Theorem \ref{result1} improves the best known bound until now, $\alpha > 1$ (due to Cowling \cite{Cowling}) down to $\alpha > 1/2$ and also the bounds for lower dimensional Zoll manifolds given by (\ref{zoll_manifold}) \cite[Theorem 6.5]{WZ}; 

\item for the fractional Schr\"odinger equations with convex phase $\psi(r)=r^a$, $a>1$, Theorem \ref{result1} improves the best known bound until now, $\alpha > a/2$ (due to Cowling \cite{Cowling}) down to $\alpha > 1/2$.
\end{enumerate}
\end{remark}

\begin{remark} \label{remark_rank2}
It is noteworthy to mention that Theorem \ref{result1} is the only result apart from Cowling's abstract result \cite{Cowling} on Riemannian symmetric spaces of rank $>1$, be it compact or non-compact type. In fact, for Riemannian symmetric spaces of compact type of rank two, for the fractional Schr\"odinger equations with convex phase $\psi(r)=r^a$, $a>2$, Theorem \ref{result1} improves the best known bound until now, $\alpha > a/2$ (due to Cowling \cite{Cowling}) down to $\alpha > 1$.
\end{remark}

\begin{remark} \label{remark_result1_key_ideas}
In the proof of Theorem \ref{result1}, the key idea is to utilize the concentration of the discrete $L^2$ spectrum of $\Delta$. In the rank one case, this follows just from the identification of the highest weight lattice as the set of non-negative integers and the growth of the eigenvalues (see Remark \ref{rank1_eigenvalue_remark}). In the rank two case, however, similar growth estimates require more work (see Lemma \ref{rank2_eigenvalue}) and  we also need the asymptotics on the number of lattice points in annuli, by appealing to the Gauss circle problem (see Lemma \ref{lattice_point_counting}).
\end{remark}

Any element of $L^2(\X)$ is naturally identified with an element of $L^2(U)$ which is right $K$-invariant. We now focus our attention to more symmetric initial data, namely elements of $L^2(U)$ which are $K$-biinvariant, that is, both left and right $K$-invariant. In the special case when rank$(\X)=1$, the algebraic notion of $K$-biinvariance has an interesting connection with the geometry of the two-point homogeneous spaces $\X$. For a fixed origin $o \in \X$, $U$ can be identified as the maximal connected group of isometries of $\X$ and $K$ as the stabilizer of $o$ in $U$. Moreover, by the two-point homogenity, $K$ fixes $o$ and acts transitively on the set of points at a given distance from $o$. Thus $K$-biinvariant functions on $U$ can be identified with functions on $\X$ radial around $o$, that is functions $f$ on $\X$ such that $f(x)$ only depends on the geodesic distance of $x$ from $o$. Henceforth, we will refer to such functions simply as radial functions.

\medskip

Due to the extra symmetry, one may expect an improvement of the sufficiency bound while specializing to $K$-biinvariant initial data, compared to the general case treated in Theorem \ref{result1}. This was recently studied for the Schr\"odinger equation on the sphere $\Sb^d$ by Chen-Duong-Lee-Yan, where they proved the sufficiency of the bound $\alpha>1/3$ for such functions \cite[Theorem 1.3]{CDLY}. Our next result generalizes their result to all rank one Riemannian symmetric spaces of compact type not only for the Schr\"odinger equation but also for other PDEs such as the Boussinesq equation and the Beam equation:
\begin{theorem} \label{result2}
Let $\X$ be a rank one Riemannian symmetric space of compact type. For the Schr\"odinger equation, the Boussinesq equation and the Beam equation, the pointwise convergence of the solution to its initial data,
\begin{equation*}
\displaystyle\lim_{t \to 0+} u(x,t)=f(x)\:,
\end{equation*}
holds for almost every $x \in \X$, whenever $f \in H^\alpha(\X)$ is $K$-biinvariant with $\alpha > 1/3$.
\end{theorem}

Theorems \ref{result1} and \ref{result2} yield sufficient conditions on the regularity threshold to guarantee pointwise convergence almost everywhere, which by standard arguments (for instance see the proof of Theorem 5 of \cite{Sjolin}) boils down to obtaining $H^\alpha(\X) \to L^p(\X)$ (for some $p \ge 1$) boundedness results for the maximal function:
\begin{equation} \label{boundedness}
\left\|\sup_{0 \le t < 2\pi} \left|e^{-it\psi\left(\sqrt{-\Delta}\right)} f \right|\right\|_{L^p(\X)} \le C \|f\|_{H^\alpha(\X)}\:,
\end{equation}
for some $C>0$. Hence to put the above results in the retrospect of the Carleson's problem, we present the following negative result which again generalizes the related result for the Schr\"odinger equation on the sphere $\Sb^d$ \cite[Theorem 1.3]{CDLY} by Chen et. al.: 
\begin{theorem} \label{result3}
Let $\X=\Sb^d,\:P^d(\R)$ ($d$ odd), $P^d(\C),\:P^d(\Hb)$ or $P^{16}(Cay)$. Then for the Schr\"odinger equation, the Boussinesq equation and the Beam equation, the maximal estimate (\ref{boundedness}) fails for (any $p \ge 1$) the class of $K$-biinvariant functions in $H^\alpha(\X)$ if $\alpha <1/4$.
\end{theorem}

\begin{remark} \label{remark_results2_3_key_ideas}
We now briefly mention the key ideas for the proofs of Theorems \ref{result2} and \ref{result3}:
\begin{enumerate}
\item For the Schr\"odinger equation, we follow the general outline of the arguments employed by Chen et. al. on $\Sb^d$:
\begin{itemize}
\item[(a)] In the proof of Theorem \ref{result2}, the maximal function is estimated near the poles by the Sobolev embedding (see Lemma \ref{sobolev_embedding}). Then away from the poles, by expanding the Jacobi polynomials, the problem is reduced to a maximal estimate on $\T$. This is then obtained by adapting the number-theoretic ideas of Bourgain \cite{Bourgain-GAFA} and Moyua-Vega \cite{MV}, relating the maximal estimate to a Strichartz estimate (see Lemmata \ref{Strichartz} and \ref{max_estimate_circle}).
\item[(b)] Following ideas of Moyua-Vega \cite{MV}, the proof of Theorem \ref{result3} uses quadratic Gauss sum evaluations to obtain the desired rate of divergence on a suitable set. Then using asymptotics of arithmetic functions, we see that the set is suitably large (see Lemma \ref{counting_lemma}).
\end{itemize}
\item The above mentioned number theoretic techniques are special for the case of the Schr\"odinger equation, as they do not seem to have a straightforward extension to other dispersive equations with non-polynomial phase, such as the Boussinesq equation or the Beam equation. Instead, we obtain Theorems \ref{result2} and \ref{result3} for the other two equations by means of a novel transference principle (see Theorem \ref{transference_principle}), which seems to be new even for $\T \cong SO(2)$ and may be of independent interest. We now explain the transference principle:
\end{enumerate}
\end{remark}

\begin{definition} \label{my_defn}
Let us consider two dispersive equations of the form (\ref{dispersive}) corresponding to $\psi_1,\psi_2$.
\begin{itemize}
\item[(i)] They are called $(\psi_1,\psi_2)$-{\it transferrable} if given that for some $\alpha_0>0$ and some $p \in [1,\infty]$, the maximal estimate 
\begin{equation*}
\left\|\sup_{0 \le t < 2\pi} \left|e^{-it\psi_1\left(\sqrt{-\Delta}\right)} f \right|\right\|_{L^p(\X)} \le C(\alpha) {\|f\|}_{H^\alpha(\X)}\:,
\end{equation*}
holds for some $C(\alpha)>0$, all $\alpha >\alpha_0$ and all $K$-biinvariant $f \in C^\infty(\X)$, it follows that the maximal estimate 
\begin{equation*}
\left\|\sup_{0 \le t < 2\pi} \left|e^{-it\psi_2\left(\sqrt{-\Delta}\right)} f \right|\right\|_{L^p(\X)} \le C'(\alpha) {\|f\|}_{H^\alpha(\X)}\:,
\end{equation*}
also holds for some $C'(\alpha)>0$, all $\alpha >\alpha_0$ and all $K$-biinvariant $f \in C^\infty(\X)$. 

\item[(ii)] If the two equations are both $(\psi_1,\psi_2)$-transferrable as well as $(\psi_2,\psi_1)$-transferrable, then they are simply called {\it transferrable}.

\item[(iii)] If there exist $R>0$ and $C>0$, such that the phase functions $\psi_1$ and $\psi_2$ satisfy
\begin{equation*}
\left|\psi_1(r)-\psi_2(r)\right| \le C\:, \text{ for all } r > R\:,
\end{equation*} 
then the equations are said to be {\it of comparable oscillation}. 
\end{itemize}
\end{definition}

\begin{remark} \label{examples_remark}
Point (iii) of Definition \ref{my_defn} means that in high frequency, both the phase functions are within bounded error. Let us now consider the equations mentioned before. The phase functions for the Boussinesq equation and the Beam equation are given respectively as,
\begin{equation*}
\psi_1(r)= r \sqrt{1+r^2} \:, \text{ and } \psi_2(r)=  \sqrt{1+r^4}\:.
\end{equation*}
For $r$ large,
\begin{eqnarray*}
&&\psi_1(r)= r \left(r + \mathcal{O}(r^{-1})\right) = r^2 + \mathcal{O}(1)\:,\\
&&\psi_2(r)= r^2 + \mathcal{O}(1)\:,
\end{eqnarray*}
and thus both are of comparable oscillation to the Schr\"odinger equation $(\psi(r)=r^2)$. 
\end{remark}

The notion of {\it comparable oscillation} defines an equivalence relation. Our next result presents the phenomenon that all the equations in the same equivalence class `behave similarly': 

\begin{theorem} \label{transference_principle}
Let $\psi_1$ and $\psi_2$ be continuous real-valued functions on $[0,\infty)$. If the dispersive equations corresponding to $\psi_1$ and $\psi_2$ are of comparable oscillation, then they are also transferrable.  
\end{theorem} 

This article is organized as follows. In Section $2$, we fix our notations, recall the essential preliminaries about Riemannian symmetric spaces of compact type and harmonic analysis thereon and also some number theoretic facts. We prove Theorems \ref{result1}, \ref{transference_principle}, \ref{result2} and \ref{result3} in Sections $3$, $4$, $5$ and $6$ respectively. Finally, we conclude by posing some open problems in Section $7$.  

\section{Preliminaries}

\subsection{Some notations} Throughout this article $c,C,\dots$ will be used to denote positive constants whose values may change on each occurrence. $\N$ will denote the set of positive integers. Two positive functions $f_1$ and $f_2$ will satisfy,
\begin{itemize}
\item  $f_1 \lesssim f_2$ if there exists $C\ge 1$ such that $f_1 \le Cf_2$;
\item  $f_1 \gtrsim f_2$ if there exists $C\ge 1$ such that $Cf_1 \ge f_2$;
\item $f_1 \asymp f_2$ if there exists $C\ge 1$ such that $\frac{1}{C}f_1 \le f_2 \le Cf_1$;
\item $f_1=o(f_2)$ if $\frac{f_1}{f_2} \to 0$ while considering some limit.
\end{itemize}
The notation $\lesssim_\varepsilon$ will also be used to denote the dependence of the parameter $\varepsilon$. In addition, if $f_1$ is complex-valued, we will write,
\begin{itemize}
\item $f_1=\mathcal{O}(f_2)$ to denote that $|f_1| \lesssim f_2$\:.
\end{itemize}
For $x \in \R$, $\lfloor x \rfloor$ will denote the largest integer smaller or equal to $x$. For a set $A$, the notation $\#A$ will denote the cardinality of $A$.

\subsection{Riemannian symmetric spaces of compact type and harmonic analysis thereon} In this subsection, we briefly recall some definitions and results on Riemannian symmetric spaces of compact type. The relevant information can be found in  \cite{HelgasonDiff,HelSymm, Helgason, OS}.

Riemannian symmetric spaces of compact type are homogeneous spaces $\X=U/K$, where $U$ is a connected, simply connected, compact, semi-simple Lie group and $K$ is a closed subgroup with the property that $U^\theta_0 \subset K \subset U^\theta$ for an involution $\theta$ of $U$. Here $U^\theta$ denotes the subgroup of $\theta$-fixed points and $U^\theta_0$ its identity component.

Let $\mathfrak{u}$ be the Lie algebra of $U$ and $\mathfrak{u}=\mathfrak{k}+\mathfrak{q}$ be the Cartan decomposition corresponding to $\theta$, where $\mathfrak{k}$ is the Lie algebra of $K$ and $\mathfrak{q}$ can be identified with the tangent space at the origin $T_o\X$ where $o=eK$. The Killing form being negative definite on $\mathfrak{u}$ induces an inner product $\langle \cdot, \cdot \rangle$ on $\mathfrak{u}$. Then $\mathfrak{k}$ and $\mathfrak{q}$ become orthogonal subspaces. We assume that the normalized Riemannian metric agrees with $\langle \cdot, \cdot \rangle$ on  $\mathfrak{q}=T_o\X$.

$\X$ is equipped with the push-forward measure of the normalized Haar measure $du$ of $U$. Let $(\delta,V_\delta)$ be an irreducible unitary representation of $U$ and $V^K_\delta$ be the space of vectors $v \in V_\delta$ fixed under $\delta(K)$. Let $\what{U}_K$ denote the collection of equivalence classes of irreducible, unitary representations $\delta$ of $U$ such that $V^K_\delta \ne 0$. Then from the Peter-Weyl theory of $U$, we can obtain the decomposition,
\begin{equation} \label{peter-weyl}
L^2(\X) = \bigoplus_{\delta \in \what{U}_K} H_\delta(\X)\:,
\end{equation}
where
\begin{equation*}
H_\delta(\X):= \left\{\langle \delta(u){\bf e}, v   \rangle \mid v \in V_\delta\right\}\:,
\end{equation*}
and ${\bf e}$ is the unique unit vector spanning $V^K_\delta$.

We now obtain a more explicit realization of  $\what{U}_K$ in terms of the highest weight lattice of the underlying root system. Let $\mathfrak{a} \subset \mathfrak{q}$ be a maximal abelian subspace, $\mathfrak{a}^*$ its dual and $\mathfrak{h} \subset \mathfrak{u}$ be a Cartan subalgebra containing $\mathfrak{a}$. Then $\mathfrak{h} = \mathfrak{h}_m \oplus \mathfrak{a}$, where $\mathfrak{h}_m = \mathfrak{h} \cap \mathfrak{k}$. The dimension of $\mathfrak{a}$ is the rank of $\X$. By duality, we can equip $\mathfrak{a}^*$ with an inner product, which is essentially induced by the Killing form. We denote by $\Sigma$ the set of non-zero restrictions to $\mathfrak{a}$ of roots of $\mathfrak{u}$ with respect to $\mathfrak{h}$.    We fix a set of positive restricted roots $\Sigma^+ \subset \Sigma$ and by $\Sigma^+_s$ denote the corresponding set of positive simple roots. We will denote by $\rho$, the half-sum of the roots in $\Sigma^+$ counted with multiplicity. The set $\Sigma^+_s$ forms a basis of $\mathfrak{a}^*$. If rank$(\X)=l$ and $\Sigma^+_s =\{\beta_1, \dots, \beta_l\}$, then \cite[Chap. VII, Lemma 2.18]{HelgasonDiff}
\begin{equation} \label{negative_killing_form}
\langle \beta_j,\beta_k \rangle \le 0\:,\:\forall 1 \le j , k \le l\:.
\end{equation}
Let $\Lambda^+(\mathfrak{h})$ denote the set of dominant integral linear functionals on $\mathfrak{h}$.  On the other hand let $\Lambda^+(U) \subset \Lambda^+(\mathfrak{h})$ denote the set of highest weights of irreducible representations of $U$. Since our $U$ is simply-connected we have $\Lambda^+(U) = \Lambda^+(\mathfrak{h})$. Let $\Lambda^+(U/K)$ denote the subset of $\Lambda^+(U)$ corresponding to $\what{U}_K$. By \cite[Theorem 4.1, p. 535]{Helgason}, we have the following identification of $\Lambda^+(U/K)$: let $\mu \in \Lambda^+(U)$, then $\mu \in \Lambda^+(U/K)$ if and only if $\mu|_{\mathfrak{h}_m}=0$ and its restriction to $\mathfrak{a}$ (that is, as an element of $\mathfrak{a}^*$) satisfies
\begin{equation*}
\frac{\langle \mu, \beta \rangle}{\langle \beta, \beta \rangle} \in \N \cup \{0\}\:,\:\: \beta \in \Sigma^+\:.
\end{equation*}
From here on we will work with the restriction of $\mu$ to $\mathfrak{a}$. Thus in view of the above identification, we can rewrite the decomposition (\ref{peter-weyl}), by an abuse of notation as,
\begin{equation*}
L^2(\X) = \bigoplus_{\mu \in \Lambda^+(U/K)} H_\mu(\X)\:.
\end{equation*}
The above decomposition can also be viewed as the eigenspace decomposition of the Laplace-Beltrami operator on the Riemannian manifold $\X$. By \cite[Lemma 1, Section 6.6, p. 367]{Pro}, for any $f \in H_\mu(\X)$,
\begin{equation*}
\Delta f = -\lambda^2_\mu f\:,
\end{equation*}
where
\begin{equation} \label{eigenvalues}
\lambda^2_\mu = \langle \mu,\mu \rangle + 2 \langle \mu, \rho \rangle\:,\:\: \mu \in \Lambda^+(U/K)\:.
\end{equation}

Let $d_\mu :=dim\left(H_\mu(\X)\right)$ and $\{Y_{\mu,j}\}_{j=1}^{d_\mu}$ be an orthonormal basis of $H_\mu(\X)$. Thus for $f \in C^\infty(\X)$, the Fourier series of $f$ is given by,
\begin{equation} \label{Fourier_series}
f=\sum_{\mu \in \Lambda^+(U/K)} \sum_{j=1}^{d_\mu} \what{f}(\mu,j)\:Y_{\mu,j}\:,
\end{equation}
where the Fourier coefficients are given by,
\begin{equation*}
\what{f}(\mu,j) = \int_{\X} f(x)\:\overline{Y_{\mu,j}(x)}\:dx\:,
\end{equation*}
and the series converges uniformly and absolutely. 

By the Plancherel identity, the Sobolev norms for the Sobolev spaces defined in (\ref{Sobolev_space}) are given by,
\begin{equation*}
\|f\|_{H^\alpha(\X)}:=\left(\sum_{\mu \in \Lambda^+(U/K)} {(1+\lambda^2_{\mu})}^\alpha \sum_{j=1}^{d_\mu} \left|\what{f}(\mu,j)\right|^2 \right)^{\frac{1}{2}}\:,\:\:\:\alpha \ge 0\:.
\end{equation*}

\subsubsection{The rank one case:} We now specialize to the rank one case. The relevant information can be found in \cite{Helgason,CT, Gigante}.

Without loss of generality, the Riemannian measure and the metric on $\X$ can be renormalized so that the total measure of $\X$ is $1$ and the diameter of $\X$ is $\pi$. For a fixed origin $o \in \X$, the points with distance from $o$ equal to the diameter of $\X$ are the antipodal points of $o$. The collection of all antipodal points of $o$ is termed as the {\em antipodal manifold}. Denoting the dimension of the antipodal manifold as $M_1$ and the dimension of $\X$ as $d$, we introduce a non-negative integer $M_2$ so that $M_1+M_2+1=d$. Furthermore, if $\Ac(\theta)$ denotes the surface measure of a geodesic sphere in $\X$, centered at $o$ with radius $\theta$, for $\theta \in [0,\pi]$, then it has the explicit form,
\begin{equation} \label{density}
\Ac(\theta)= C {\left(\sin \frac{\theta}{2}\right)}^{M_1} {(\sin \theta)}^{M_2}\:,
\end{equation}
where the constant $C>0$ is such that $\int_0^\pi \Ac(\theta) d\theta=1$. The specific values of the parameters $M_1,M_2$ are given in Table \ref{table:1}.
 
\begingroup

\setlength{\tabcolsep}{10pt} 
\renewcommand{\arraystretch}{1.5} 
\begin{table}[H]
\centering
\begin{tabular}{ ||p{1.7cm}|| p{1cm} | p{1cm} || p{1.7cm}|| p{1cm} | p{1cm}| }
 \hline
 $\X$ & $M_1$ & $M_2$ & $\X$ & $M_1$ & $M_2$ \\
 \hline 
$\Sb^d$   & $0$   & $d-1$ & $P^d(\Hb)$ & $d-4$ & \:\:$3$  \\
$P^d(\R)$ & $d-1$   & \:\:$0$ & $P^{16}(Cay)$ & \:\:$8$ & \:\:$7$ \\
$P^d(\C)$ & $d-2$  & \:\:$1$ & & &   \\
\hline
\end{tabular}

\medskip

\caption{}
\label{table:1}
\end{table}
\endgroup

In rank one, we have the concrete realization of the set $\Lambda^+(U/K)=\N \cup \{0\}$ and the eigenvalues of the Laplace-Beltrami operator $\Delta$ are parametrized as $-\lambda^2_n$ where
\begin{equation} \label{rank1_eigenvalue_form} 
\lambda^2_n =n(n+\sigma+\tau+1)\:,\:\:n \in \N \cup \{0\}\:,
\end{equation}  
and the parameters $\sigma, \tau$ and the eigenvalues $\lambda^2_n$ are given in Table \ref{table:2}.

\begingroup

\setlength{\tabcolsep}{10pt} 
\renewcommand{\arraystretch}{1.5} 
\begin{table}[h!]
\centering
\begin{tabular}{ ||p{1.7cm}|| p{1cm}| p{1cm}| p{2.3cm}|  }
 \hline
 $\X$ & $\sigma$ & $\tau$ & $\lambda^2_n$\\
 \hline 
$\Sb^d$   & $\frac{d-2}{2}$   & $\frac{d-2}{2}$ & $n(n+d-1)$ \\
$P^d(\R)$ & $\frac{d-2}{2}$   & $-\frac{1}{2}$  & $n\left(n+\frac{d-1}{2}\right)$ \\
$P^d(\C)$ & $\frac{d-2}{2}$  & \:\:$0$  & $n\left(n+\frac{d}{2}\right)$ \\
$P^d(\Hb)$ & $\frac{d-2}{2}$ & \:\:$1$ & $n\left(n+1+\frac{d}{2}\right)$   \\
$P^{16}(Cay)$ & \:\:$7$ & \:\:$3$  & $n(n+11)$ \\
\hline
\end{tabular}

\medskip

\caption{}
\label{table:2}
\end{table}
\endgroup

\begin{remark} \label{rank1_eigenvalue_remark}
From Table \ref{table:2} we observe that 
\begin{equation*}
\lambda^2_n \asymp n^2\:,\:\: n \in \N \cup \{0\}\:,
\end{equation*}
with the implicit constant only depending on the intrinsic geometry of $\X$.
\end{remark}

The quantity $d_n$, the dimension of the finite dimensional subspaces $H_n$ satisfy,
\begin{equation} \label{dimension_growth}
d_n \asymp \left(1+n\right)^{d-1}\:.
\end{equation}
For $f \in C^\infty(\X)$, by taking the orthogonal projections $Y_nf$ onto $H_n$, the Fourier series (\ref{Fourier_series}) can be rewritten as
\begin{equation} \label{dummy_series}
f(x)=\sum_{n=0}^\infty Y_nf(x)\:,
\end{equation}
where 
\begin{eqnarray*}
 Y_nf(x) &=& \sum_{j=1}^{d_n} \what{f}(n,j)\:Y_{n,j}(x) \\
 &=& \int_{\X} f(y)\left(\:\sum_{j=1}^{d_n} Y_{n,j}(x)\overline{Y_{n,j}(y)}\right)\:dy \\
 &=& \int_{\X} f(y)\:Z_n(x,y)\:dy\:,
\end{eqnarray*}
where 
\begin{equation*}
Z_n(x,y):=\sum_{j=1}^{d_n} Y_{n,j}(x)\:\overline{Y_{n,j}(y)}\:,
\end{equation*}
are the {\em zonal spherical functions} of degree $n$ and pole $x$. In rank one, these functions are given in terms of Jacobi polynomials:
\begin{equation} \label{in_terms_jacobi}
Z_n(x,y):= d_n \frac{P^{(\sigma,\tau)}_n{\left(\cos(d(x,y))\right)}}{P^{(\sigma,\tau)}_n(1)}\:,\:\:x,y \in \X\:,
\end{equation}
We also note that as $P^{(\sigma,\tau)}_0$ is a constant function, so is the zonal spherical function of degree $0$, $Z_0(\cdot,\cdot)$. The following expansion of the Jacobi polynomials will be relevant for us:
\begin{lemma} \cite[Theorem 8.21.13]{Szego}\label{jacobi_polynomial}
Let $\sigma >-1$ and $\tau>-1$. Then we have, for $n \in \N$,
\begin{eqnarray*}
&&P^{(\sigma,\tau)}_n (\cos \theta)\\
&=& \frac{1}{(\pi n)^{\frac{1}{2}} \left(\sin \frac{\theta}{2}\right)^{\sigma +\frac{1}{2}}\left(\cos \frac{\theta}{2}\right)^{\tau +\frac{1}{2}}}\left\{\cos\left(\left(n+\frac{\sigma+\tau+1}{2}\right)\theta -\frac{\left(\sigma +\frac{1}{2}\right)}{2}\pi\right)+\frac{\mathcal{O}(1)}{n\sin\theta}\right\}\:,
\end{eqnarray*}
for $cn^{-1} \le \theta \le \pi -cn^{-1}$, where $c$ is a fixed positive number.
\end{lemma}
 
Clearly $Z_n(o,\cdot)$ is a radial function. It is an eigenfunction of the Laplace-Beltrami operator with eigenvalue $-\lambda^2_n$. Furthermore, by \cite[Remark 3.5]{Pasquale}, we have
\begin{equation}\label{L2_Zn}
\|Z_n(o,\cdot)\|_{L^2(\X)}=d^{\frac{1}{2}}_n\:.
\end{equation}
We note that by the Cartan decomposition and the transitive action of $K$ on the set of points at a given distance from $o$, $K$-biinvariant functions on $U$ can be identified as radial functions on $\X$ and hence as functions on $[0,\pi]$. Then for $n \in \N$, identifying the function $\theta \mapsto P^{(\sigma,\tau)}_n{\left(\cos(\theta)\right)}$ on $[0,\pi]$ as a $K$-biinvariant function on $\X$, from (\ref{in_terms_jacobi}) by combining (\ref{dimension_growth}), (\ref{L2_Zn}) and the fact that $$P^{(\sigma,\tau)}_n(1) = {n+\sigma \choose n} \asymp n^\sigma=n^{\frac{d-2}{2}}\:,$$
we get
\begin{equation} \label{jacobi_poly_L2}
\left\|P^{(\sigma,\tau)}_n{\left(\cos(\cdot)\right)}\right\|_{L^2(\X)} \asymp n^{-\frac{1}{2}}\:,\:n \in \N\:.
\end{equation}
When $f$ is also $K$-biinvariant or equivalently radial, so are its orthogonal projections $Y_nf$ and are given by
\begin{equation} \label{dummy_coefficient}
Y_nf=b_nZ_n(o,\cdot)\:,
\end{equation}
for some $b_n \in \C$. 

Now renormalizing the zonal spherical functions as,
\begin{equation} \label{normalization}
\tilde{Z}_n := \frac{Z_n(o,\cdot)}{\:\:\|Z_n(o,\cdot)\|_{L^2}}\:,
\end{equation}
we rewrite (\ref{dummy_series}) and (\ref{dummy_coefficient}), in terms of $\tilde{Z}_n$ to obtain,
\begin{equation} \label{spherical_fourier_series}
f=\sum_{n=0}^\infty a_n\:\tilde{Z}_n\:,
\end{equation}
for some $a_n \in \C$\:. We will refer to (\ref{spherical_fourier_series}) as the spherical Fourier series of $f$. 

By the Plancherel identity, the Sobolev norms for $K$-biinvariant functions $f$ are given by,
\begin{equation*}
\|f\|_{H^\alpha(\X)}:=\left(\sum_{n=0}^\infty {(1+\lambda^2_n)}^\alpha \left|a_n\right|^2 \right)^{\frac{1}{2}}\:,\:\:\:\alpha \ge 0\:.
\end{equation*} 

We will also require some efficient $L^\infty$ bound of $\tilde{Z}_n$:
\begin{lemma} \label{Linfty_bound}
\begin{itemize}
\item[(i)] Near the origin, we have
 \begin{equation*}
 \|\tilde{Z}_n\|_{L^\infty\left(\left[0,\frac{\pi}{2}\right]\right)} \lesssim {\left(1+n\right)}^{\frac{d-1}{2}}\:,\:\: n \in \N \cup \{0\}\:.
\end{equation*}  
\item[(ii)] Near the antipodal manifold, we have
\begin{equation*}
\|\tilde{Z}_n\|_{L^\infty\left(\left[\frac{\pi}{2},\pi\right]\right)} \lesssim \begin{cases}
(1+n)^{\frac{d-1}{2}}\:\:\:&\text{if } \X=\Sb^d\:,\\
1\:\:\:&\text{if } \X=P^d(\R)\:,\\
(1+n)^{\frac{1}{2}}\:\:\:&\text{if } \X=P^d(\C)\:,\\
(1+n)^{\frac{3}{2}}\:\:\:&\text{if } \X=P^d(\Hb)\:,\\
(1+n)^{\frac{7}{2}}\:\:\:&\text{if } \X=P^{16}(Cay)\:.
\end{cases}
\end{equation*}
\end{itemize}
\end{lemma}
\begin{proof}
The result follows from (\ref{normalization}), (\ref{L2_Zn}), (\ref{dimension_growth}) and the pointwise bounds in \cite[Lemma 7]{CT}.
\end{proof}

\subsection{Number theoretic facts}
In this subsection, we recall some results in number theory, which will be crucial for us. All of this can be found in the classical text \cite{Hardy}.

We first see the following classical asymptotics on the Gauss circle problem, i.e. the lattice point counting inside disks, centered at the origin:
\begin{lemma} \label{lattice_point_counting} For $n \in \N$, let $$D_n:= \left\{(n_1,n_2) \in \Z^2 \mid \left(n^2_1 + n^2_2\right)^\frac{1}{2} \le n\right\}\:.$$ Then 
$$\#D_n = \pi n^2 (1+o(1))\:,\text{ as } n \to \infty\:.$$
\end{lemma}
Next we look at some fundamental arithmetic functions, that is, functions defined on $\N$. The Euler's totient function $\varphi$ is defined as $\varphi(n):=$ the number of positive integers less than and prime to $n$. The M\"obius function $\mu$ is defined as,
\begin{equation*}
\mu(n):=\begin{cases}
1\:\:\:&\text{if } n=1\:,\\
{(-1)}^j\:\:\:&\text{if } n \text{ is the product of } j \text{ distinct prime factors}\:, \\
0\:\:\:&\text{if } n \text{ has a squared factor }\:.
\end{cases}
\end{equation*} 
The following is the fundamental identity satisfied by the M\"obius function \cite[Theorem 263]{Hardy}:
\begin{equation} \label{mobius_identity}
\sum_{m|n} \mu(m) =\begin{cases}
1\:\:\:&\text{ if } n=1\:,\\
0\:\:\:&\text{ if } n>1\:.
\end{cases}
\end{equation}
The functions $\varphi$ and $\mu$ are related via the M\"obius inversion \cite[Eqn. 16.3.1]{Hardy}:
\begin{equation} \label{mobius_inversion}
\frac{\varphi(n)}{n}= \sum_{m|n} \frac{\mu(m)}{m}\:.
\end{equation}
The divisor function is defined as $d(n):=$ the number of divisors of $n$ including $1$ and $n$.
We have the following asymptotics of the functions $\varphi$ and $d$ \cite[Theorems 315 and 327]{Hardy}: for any $\varepsilon>0$,
\begin{equation} \label{asymptotics}
d(n) \lesssim_{\varepsilon} n^\varepsilon \text{ and } \varphi(n) \gtrsim_{\varepsilon} n^{1-\varepsilon}\:,\text{ as } n \to \infty\:.
\end{equation}

\section{General result for symmetric spaces of rank $1$ and $2$}
In this section, we prove Theorem \ref{result1}. We first start off with a linear algebraic result:
\begin{lemma} \label{linear_algebra_lemma}
Let $\left(V,\langle \cdot,\cdot \rangle\right)$ be a 2-dimensional real inner product space with a basis $\{v_1,v_2\}$. Let 
\begin{equation*}
\xi_j(v):= \frac{\langle v,v_j \rangle}{\langle v_j,v_j \rangle}\:,\: v \in V\:,\:j=1,2\:.
\end{equation*}
Then for any $v \in V$,
\begin{equation*}
\langle v,v \rangle = \frac{\xi_1(v)^2 \langle v_1,v_1 \rangle^2 \langle v_2,v_2 \rangle + \xi_2(v)^2 \langle v_2,v_2 \rangle^2 \langle v_1,v_1 \rangle - 2\xi_1(v) \xi_2(v) \langle v_1,v_1 \rangle \langle v_1,v_2 \rangle \langle v_2,v_2 \rangle }{\langle v_1,v_1 \rangle \langle v_2,v_2 \rangle - \langle v_1,v_2 \rangle^2}\:.
\end{equation*}
\end{lemma}
\begin{proof}
We first choose and fix $v \in V$. There exist $c_1(v),c_2(v) \in \R$ such that
\begin{equation}  \label{linear_eq1}
v=c_1(v)v_1+c_2(v)v_2\:.
\end{equation}
Thus we have,
\begin{equation} \label{linear_eq2}
\langle v,v \rangle = c_1(v)^2 \langle v_1,v_1 \rangle + 2c_1(v)c_2(v) \langle v_1,v_2 \rangle + c_2(v)^2 \langle v_2,v_2 \rangle  \:.
\end{equation}
From (\ref{linear_eq1}), we also have
\begin{eqnarray} \label{linear_eq3}
\langle v,v_1 \rangle &=& c_1(v) \langle v_1,v_1 \rangle + c_2(v) \langle v_1,v_2 \rangle\:,\\
\langle v,v_2 \rangle &=& c_1(v) \langle v_1,v_2 \rangle + c_2(v) \langle v_2,v_2 \rangle\:. \label{linear_eq4}
\end{eqnarray}
Thus combining (\ref{linear_eq2})-(\ref{linear_eq4}) and invoking the definition of $\xi_j,\:j=1,2$, we get
\begin{equation} \label{linear_eq5}
\langle v,v \rangle = c_1(v)\xi_1(v)\langle v_1,v_1 \rangle + c_2(v)\xi_2(v)\langle v_2,v_2 \rangle \:.
\end{equation}
Next solving the equations (\ref{linear_eq3}) and (\ref{linear_eq4}), we have
\begin{eqnarray} \label{linear_eq6}
c_1(v) &=& \frac{\xi_1(v) \langle v_1,v_1 \rangle \langle v_2,v_2 \rangle - \xi_2(v) \langle v_2,v_2 \rangle \langle v_1,v_2 \rangle }{\langle v_1,v_1 \rangle \langle v_2,v_2 \rangle - \langle v_1,v_2 \rangle^2}\:, \\
c_2(v) &=& \frac{\xi_2(v) \langle v_1,v_1 \rangle \langle v_2,v_2 \rangle - \xi_1(v) \langle v_1,v_1 \rangle \langle v_1,v_2 \rangle }{\langle v_1,v_1 \rangle \langle v_2,v_2 \rangle - \langle v_1,v_2 \rangle^2}\:. \label{linear_eq7}
\end{eqnarray}
The result now follows by plugging (\ref{linear_eq6}) and (\ref{linear_eq7}) in (\ref{linear_eq5}).
\end{proof}

As an application of Lemma \ref{linear_algebra_lemma}, we obtain the following asymptotic lower bound for eigenvalues of $-\Delta$ for rank $2$:
\begin{lemma} \label{rank2_eigenvalue}
Let $\X$ be a Riemannian symmetric space of compact type of rank $2$. Then the eigenvalues of $-\Delta$ satisfy
\begin{equation*}
\lambda_\mu^2 \gtrsim n_1(\mu)^2 + n_2(\mu)^2\:,\:\: \mu  \in \Lambda^+(U/K)\:,
\end{equation*}
where the implicit constant only depends on the intrinsic geometry of $\X$ and the coefficients $n_j$ are defined in terms of the inner product induced by the Killing form as,
\begin{equation*}
n_j(\mu):= \frac{\langle \mu, \beta_j \rangle}{\langle \beta_j, \beta_j \rangle} \in \N \cup \{0\}\:,\:\:j=1,2,\:\text{ where }\: \{\beta_1, \beta_2\} = \Sigma^+_s\:.
\end{equation*}
\end{lemma}
\begin{proof}
For $\mu \in \Lambda^+(U/K)$, we have by (\ref{eigenvalues}),
\begin{equation*}
\lambda^2_\mu = \langle \mu,\mu \rangle + 2 \langle \mu, \rho \rangle \ge \langle \mu,\mu \rangle\:.
\end{equation*}
Now $\{\beta_1,\beta_2\}$ forms a basis of $\mathfrak{a}^*$ where $\{\beta_1, \beta_2\} = \Sigma^+_s$. Then by Lemma \ref{linear_algebra_lemma}, we have
\begin{equation*}
\langle \mu,\mu \rangle = \frac{n_1(\mu)^2 \langle \beta_1,\beta_1 \rangle^2 \langle \beta_2,\beta_2 \rangle + n_2(\mu)^2 \langle \beta_2,\beta_2 \rangle^2 \langle \beta_1,\beta_1 \rangle - 2n_1(\mu) n_2(\mu) \langle \beta_1,\beta_1 \rangle \langle \beta_1,\beta_2 \rangle \langle \beta_2,\beta_2 \rangle }{\langle \beta_1,\beta_1 \rangle \langle \beta_2,\beta_2 \rangle - \langle \beta_1,\beta_2 \rangle^2}\:.
\end{equation*}
Now by (\ref{negative_killing_form}), as $\langle \beta_1, \beta_2 \rangle \le 0$, we have 
\begin{equation*}
\langle \mu,\mu \rangle \ge n_1(\mu)^2\frac{ \langle \beta_1,\beta_1 \rangle^2 \langle \beta_2,\beta_2 \rangle}{\langle \beta_1,\beta_1 \rangle \langle \beta_2,\beta_2 \rangle - \langle \beta_1,\beta_2 \rangle^2} + n_2(\mu)^2 \frac{\langle \beta_2,\beta_2 \rangle^2 \langle \beta_1,\beta_1 \rangle }{\langle \beta_1,\beta_1 \rangle \langle \beta_2,\beta_2 \rangle - \langle \beta_1,\beta_2 \rangle^2}\:.
\end{equation*} 
This completes the proof of Lemma \ref{rank2_eigenvalue}.
\end{proof}

We are now in a position to present the proof of Theorem \ref{result1}.
\begin{proof}[Proof of Theorem \ref{result1}]
It suffices to prove that the maximal estimate
\begin{equation} \label{result1_pf_eq1}
\left\|\sup_{0 \le t < 2\pi} \left|e^{-it\psi\left(\sqrt{-\Delta}\right)} f \right|\right\|_{L^2(\X)} \lesssim \|f\|_{H^\alpha(\X)}\:,
\end{equation}
holds for all 
\begin{equation} \label{alpha_condition} 
\alpha > \begin{cases}
	 \frac{1}{2}  &\text{ if } rank(\X)=1\:, \\
	 1  &\text{ if } rank(\X)=2\:,
	\end{cases}
\end{equation}
and all $f \in C^\infty(\X)$. To prove (\ref{result1_pf_eq1}), it suffices to consider its linearization,
\begin{equation*}
T_\psi f(x):= e^{-it(x)\psi\left(\sqrt{-\Delta}\right)} f(x)\:,\:\: x \in \X\:,
\end{equation*}
where $t(\cdot):\X \to [0,2\pi)$ is a measurable function and obtain the boundedness result
\begin{equation} \label{result1_pf_eq2}
\left\|T_\psi f\right\|_{L^2(\X)} \lesssim \|f\|_{H^\alpha(\X)}\:,
\end{equation}
for all $\alpha$ satisfying (\ref{alpha_condition}) and all $f \in C^\infty(\X)$. For $f \in C^\infty(\X)$, by the expression (\ref{Fourier_series}), we have
\begin{equation*}
T_\psi f(x) = \sum_{\mu \in \Lambda^+(U/K)} e^{-it(x)\psi\left(\lambda_\mu\right)}\sum_{j=1}^{d_\mu} \what{f}(\mu,j)\:Y_{\mu,j}(x)\:,\:\: x \in \X\:. 
\end{equation*}
Thus
\begin{eqnarray*}
\left\|T_\psi f\right\|_{L^2(\X)} &=& \left\|\sum_{\mu \in \Lambda^+(U/K)} e^{-it(\cdot)\psi\left(\lambda_\mu\right)}\sum_{j=1}^{d_\mu} \what{f}(\mu,j)\:Y_{\mu,j}\right\|_{L^2(\X)} \\
& \le & \sum_{\mu \in \Lambda^+(U/K)} \left\|\sum_{j=1}^{d_\mu} \what{f}(\mu,j)\:Y_{\mu,j}\right\|_{L^2(\X)}\:.
\end{eqnarray*}
Now by the Plancherel theorem,
\begin{equation*}
\sum_{\mu \in \Lambda^+(U/K)} \left\|\sum_{j=1}^{d_\mu} \what{f}(\mu,j)\:Y_{\mu,j}\right\|_{L^2(\X)} = \sum_{\mu \in \Lambda^+(U/K)} \left(\sum_{j=1}^{d_\mu} \left|\what{f}(\mu,j)\right|^2\right)^{\frac{1}{2}}\:.
\end{equation*}
Then by the Cauchy-Schwarz inequality, we have for any $\alpha \ge 0$,
\begin{eqnarray*}
&&\sum_{\mu \in \Lambda^+(U/K)} \left(\sum_{j=1}^{d_\mu} \left|\what{f}(\mu,j)\right|^2\right)^{\frac{1}{2}}\\
&=& \sum_{\mu \in \Lambda^+(U/K)} {(1+\lambda^2_{\mu})}^{-\frac{\alpha}{2}} \left({(1+\lambda^2_{\mu})}^\alpha \sum_{j=1}^{d_\mu} \left|\what{f}(\mu,j)\right|^2\right)^{\frac{1}{2}} \\
& \le & \left(\sum_{\mu \in \Lambda^+(U/K)} {(1+\lambda^2_{\mu})}^{-\alpha}\right)^{\frac{1}{2}} \left(\sum_{\mu \in \Lambda^+(U/K)} {(1+\lambda^2_{\mu})}^\alpha \sum_{j=1}^{d_\mu} \left|\what{f}(\mu,j)\right|^2 \right)^{\frac{1}{2}} \\
&=& \left(\sum_{\mu \in \Lambda^+(U/K)} {(1+\lambda^2_{\mu})}^{-\alpha}\right)^{\frac{1}{2}} \|f\|_{H^\alpha(\X)}\:.
\end{eqnarray*}
Thus to prove (\ref{result1_pf_eq2}), it suffices to prove that for $\alpha$ satisfying (\ref{alpha_condition}), we have
\begin{equation} \label{result1_pf_eq3}
\sum_{\mu \in \Lambda^+(U/K)} {(1+\lambda^2_{\mu})}^{-\alpha} < \infty\:\:.
\end{equation}
Now in the case when rank$(\X)=1$, $\Lambda^+(U/K)=\N \cup \{0\}$ and by Remark \ref{rank1_eigenvalue_remark}, the eigenvalues satisfy $\lambda^2_n \gtrsim n^2\:,\: n \in \N \cup \{0\}$, with the implicit constant only depending on the intrinsic geometry of $\X$. Thus we get that
\begin{equation*}
\sum_{\mu \in \Lambda^+(U/K)} {(1+\lambda^2_{\mu})}^{-\alpha} \lesssim 1+\sum_{n=1}^\infty \frac{1}{n^{2\alpha}}\:,
\end{equation*}
and the last sum converges if $\alpha >1/2$. This yields (\ref{result1_pf_eq3}) when rank$(\X)=1$.

\medskip

For the case when rank$(\X)=2$, by Lemma \ref{rank2_eigenvalue},
\begin{eqnarray*}
\sum_{\mu \in \Lambda^+(U/K)} {(1+\lambda^2_{\mu})}^{-\alpha} &\lesssim & \sum_{\mu \in \Lambda^+(U/K)} \frac{1}{{(1+n_1(\mu)^2+n_2(\mu)^2)}^{\alpha}} \\
& \le & \sum_{(n_1,n_2) \in \left(\N \cup \{0\}\right)^2} \frac{1}{{(1+n^2_1+n^2_2)}^{\alpha}}\\
& \le & 1+ \sum_{(n_1,n_2) \in \left(\N \cup \{0\}\right)^2 \setminus \{(0,0)\}} \int_{{(n^2_1+n^2_2)}^{1/2}-1}^{{(n^2_1+n^2_2)}^{1/2}} \frac{dr}{{(1+r^2)}^\alpha}\:.
\end{eqnarray*}
We now consider for $r>0$, the set
\begin{equation*}
\mathscr{A}_{r,r+1} :=\left\{(n_1,n_2) \in \left(\N \cup \{0\}\right)^2 \mid r \le {(n^2_1+n^2_2)}^{\frac{1}{2}} \le r+1\right\}\:.
\end{equation*}
Thus by Lemma \ref{lattice_point_counting}, we get for all $r \ge 1$,
\begin{equation} \label{result1_pf_eq4}
\#\mathscr{A}_{r,r+1}  \lesssim r\:\:.
\end{equation}
Hence by the Fubini-Tonelli theorem and the estimate (\ref{result1_pf_eq4}), we get that
\begin{eqnarray*}
\sum_{(n_1,n_2) \in \left(\N \cup \{0\}\right)^2 \setminus \{(0,0)\}} \int_{{(n^2_1+n^2_2)}^{1/2}-1}^{{(n^2_1+n^2_2)}^{1/2}} \frac{dr}{{(1+r^2)}^\alpha} &=& \int_0^\infty \sum_{(n_1,n_2) \in \mathscr{A}_{r,r+1}} \frac{dr}{{(1+r^2)}^\alpha} \\
&=&  \int_0^\infty  \frac{\left(\#\mathscr{A}_{r,r+1}\right)dr}{{(1+r^2)}^\alpha} \\
& \lesssim & 1 \:+\:\int_1^\infty  \frac{r\:dr}{{(1+r^2)}^{\alpha}} \:. 
\end{eqnarray*}
The last integral converges if $\alpha>1$. Thus for $\alpha > 1$, we have (\ref{result1_pf_eq3}) and this completes the proof of Theorem \ref{result1}.
\end{proof}

\section{The transference principle}
From this point onwards, we will be working with $K$-biinvariant data on rank one Riemannian symmetric spaces of compact type $\X$. In this section, we will prove the transference principle Theorem \ref{transference_principle}, but first we see the following result on different scales of Sobolev norms:
\begin{lemma} \label{sobolev_comparison}
Let $f \in C^\infty(\X)$ with spherical Fourier series $f=\displaystyle\sum_{n=N}^\infty a_n \tilde{Z}_n$ (as in (\ref{spherical_fourier_series})), for some $N \in \N$. Then for $0 \le \beta_1 < \beta_2 < \infty$, we have
\begin{equation*}
\|f\|_{H^{\beta_1}(X)} \lesssim N^{-(\beta_2-\beta_1)}\:\|f\|_{H^{\beta_2}(X)}\:,
\end{equation*}
with the implicit constant depending only on the intrinsic geometry of $\X$.
\end{lemma}
\begin{proof}
The result follows from a straightforward computation. Indeed,
\begin{eqnarray*}
\|f\|_{H^{\beta_1}(X)} &=& \left(\sum_{n=N}^\infty {\left(1+\lambda^2_n\right)}^{\beta_1}\:{|a_n|}^2\:\right)^{\frac{1}{2}} \\
&=& \left(\sum_{n=N}^\infty {\left(1+\lambda^2_n\right)}^{-(\beta_2-\beta_1)}\:{\left(1+\lambda^2_n\right)}^{\beta_2}\:{|a_n|}^2\:\right)^{\frac{1}{2}} \\
&\le & \left(\sum_{n=N}^\infty \lambda^{-2(\beta_2-\beta_1)}_n\:{\left(1+\lambda^2_n\right)}^{\beta_2}\:{|a_n|}^2\:\right)^{\frac{1}{2}}\:.
\end{eqnarray*}
Now using Remark \ref{rank1_eigenvalue_remark}, we get that
\begin{eqnarray*}
\left(\sum_{n=N}^\infty \lambda^{-2(\beta_2-\beta_1)}_n\:{\left(1+\lambda^2_n\right)}^{\beta_2}\:{|a_n|}^2\:\right)^{\frac{1}{2}} &\lesssim & \left(\sum_{n=N}^\infty n^{-2(\beta_2-\beta_1)}\:{\left(1+\lambda^2_n\right)}^{\beta_2}\:{|a_n|}^2\:\right)^{\frac{1}{2}} \\
& \le & N^{-(\beta_2-\beta_1)}\:\|f\|_{H^{\beta_2}(X)}\:.
\end{eqnarray*}
\end{proof}
We now present the proof of Theorem \ref{transference_principle}.
\begin{proof}[Proof of Theorem \ref{transference_principle}]
Without loss of generality, let us assume that for some $\alpha_0>0$ and some $p \in [1,\infty]$, the maximal estimate 
\begin{equation} \label{transference_pf_eq1}
\left\|\sup_{0 \le t < 2\pi} \left|e^{-it\psi_1\left(\sqrt{-\Delta}\right)} f \right|\right\|_{L^p(\X)} \lesssim {\|f\|}_{H^\alpha(\X)}\:,
\end{equation}
holds for all $\alpha >\alpha_0$ and all $K$-biinvariant $f \in C^\infty(\X)$. We have to show that the maximal estimate 
\begin{equation} \label{transference_pf_eq2}
\left\|\sup_{0 \le t < 2\pi} \left|e^{-it\psi_2\left(\sqrt{-\Delta}\right)} f \right|\right\|_{L^p(\X)} \lesssim {\|f\|}_{H^\alpha(\X)}\:,
\end{equation}
also holds for all $\alpha >\alpha_0$ and all $K$-biinvariant $f \in C^\infty(\X)$.

By the hypothesis there exist positive constants $R$ and $C$ such that 
\begin{equation} \label{transference_pf_eq3}
\left|\psi_1(r)-\psi_2(r)\right| \le C\:,\:\: r \in (R, \infty)\:.
\end{equation} 
Let us choose and fix $f \in C^\infty(\X)$. Correspondingly, we have the spherical Fourier series (as in (\ref{spherical_fourier_series})),
\begin{equation*}
f = \sum_{n=0}^\infty a_n \tilde{Z}_n\:.
\end{equation*}
Let $m_0$ be the smallest positive integer such that $2^{m_0}> R$. We decompose $f$ as follows,
\begin{equation*} 
f=f_{m_0} + \sum_{m=m_0+1}^\infty f_m\:,
\end{equation*}
where 
\begin{equation*}
f_{m_0}:= \sum_{n=0}^{2^{m_0+1}-1} a_n \tilde{Z}_n \:, \text{ and for } m \ge m_0+1,\:\:\: f_m :=  \sum_{n=2^m}^{2^{m+1}-1} a_n \tilde{Z}_n\:.
\end{equation*}
Invoking the above decomposition of $f$, we get that
\begin{eqnarray} \label{transference_pf_eq4}
&&\left\|\sup_{0 \le t < 2\pi} \left|e^{-it\psi_2\left(\sqrt{-\Delta}\right)} f \right|\right\|_{L^p(\X)}\nonumber\\
 & \le & \left\|\sup_{0 \le t < 2\pi} \left|e^{-it\psi_2\left(\sqrt{-\Delta}\right)} f_{m_0} \right|\right\|_{L^p(\X)} + \sum_{m=m_0+1}^\infty \left\|\sup_{0 \le t < 2\pi} \left|e^{-it\psi_2\left(\sqrt{-\Delta}\right)} f_m \right|\right\|_{L^p(\X)}\:.
\end{eqnarray}
Since $f$ is $K$-biinvariant, so is each of its pieces $f_m$ above, their propagations $e^{-it\psi_j\left(\sqrt{-\Delta}\right)} f_m$, for $j=1,2$ and also the corresponding maximal functions.  

By the $L^\infty$ bound given in (\ref{Linfty_bound}) and the Cauchy-Schwarz inequality we have for all $t \in [0,2\pi)$, all $\theta \in [0,\pi]$ and all $\alpha>\alpha_0$,
\begin{eqnarray*}
\left|e^{-it\psi_2\left(\sqrt{-\Delta}\right)} f_{m_0}(\theta) \right| &=& \left|\sum_{n=0}^{2^{m_0+1}-1} e^{-it\psi_2\left(\lambda_n\right)} a_n \tilde{Z}_n \right| \\
&\le & \sum_{n=0}^{2^{m_0+1}-1} |a_n|\:\left|\tilde{Z}_n(\theta)\right|\\
&\lesssim & \|f_{m_0}\|_{H^{\alpha}(\X)} \\
& \le & \|f\|_{H^{\alpha}(\X)}\:,
\end{eqnarray*}
from which it follows that
\begin{equation} \label{transference_pf_eq5}
\left\|\sup_{0 \le t < 2\pi} \left|e^{-it\psi_2\left(\sqrt{-\Delta}\right)} f_{m_0} \right|\right\|_{L^p(\X)} \lesssim  \|f\|_{H^{\alpha}(\X)}\:.
\end{equation}
Next for each $m \ge m_0+1$, we have for all $\theta \in [0,\pi]$ and all $t \in [0,2\pi)$,
\begin{eqnarray*}
&&\left|e^{-it\psi_1\left(\sqrt{-\Delta}\right)} f_m(\theta)   - e^{-it\psi_2\left(\sqrt{-\Delta}\right)} f_m(\theta) \right| \\
&=& \left|\sum_{n=2^m}^{2^{m+1}-1} \left(e^{-it\psi_1\left(\lambda_n\right)} - e^{-it\psi_2\left(\lambda_n\right)}\right) a_n \tilde{Z}_n(\theta) \right|\\
&=& \left|\sum_{n=2^m}^{2^{m+1}-1} e^{-it\psi_1\left(\lambda_n\right)}\left(1 - e^{-it\left(\psi_2\left(\lambda_n\right)-\psi_1\left(\lambda_n\right)\right)}\right) a_n \tilde{Z}_n(\theta) \right|\:.
\end{eqnarray*}
Now using the smoothness of the exponential in time, we expand the exponential in the Taylor series to get,
\begin{eqnarray*}
&&\left|\sum_{n=2^m}^{2^{m+1}-1} e^{-it\psi_1\left(\lambda_n\right)}\left(1 - e^{-it\left(\psi_2\left(\lambda_n\right)-\psi_1\left(\lambda_n\right)\right)}\right) a_n \tilde{Z}_n(\theta) \right| \\
&=& \left|\sum_{n=2^m}^{2^{m+1}-1} e^{-it\psi_1\left(\lambda_n\right)}\left[1 - \sum_{j=0}^\infty \frac{\left\{-it\left(\psi_2\left(\lambda_n\right)-\psi_1\left(\lambda_n\right)\right)\right\}^j}{j!}\right]
a_n \tilde{Z}_n(\theta) \right|\\
& \le & \sum_{j=1}^\infty \frac{t^j}{j!} \left|\sum_{n=2^m}^{2^{m+1}-1}e^{-it\psi_1\left(\lambda_n\right)} \left[\psi_2(\lambda_n)-\psi_1(\lambda_n)\right]^j a_n\:\tilde{Z}_n(\theta)\right|\:.
\end{eqnarray*}
For each $j \in \N$ and $n \in \left\{2^m, \dots, 2^{m+1}-1\right\}$, setting $a_{j,n}:= \left[\psi_2(\lambda_n)-\psi_1(\lambda_n)\right]^j a_n$, we define,
\begin{equation*}
g_{j,m}:= \sum_{n=2^m}^{2^{m+1}-1} a_{j,n} \tilde{Z}_n\;.
\end{equation*}
Then the computation above yields, for each $m \ge m_0+1$, for all $\theta \in [0,\pi]$ and all $t \in [0,2\pi)$,
\begin{equation*}
\left|e^{-it\psi_1\left(\sqrt{-\Delta}\right)} f_m(\theta)   - e^{-it\psi_2\left(\sqrt{-\Delta}\right)} f_m(\theta) \right| \le  \sum_{j=1}^\infty \frac{t^j}{j!} \left|e^{-it\psi_1\left(\sqrt{-\Delta}\right)} g_{j,m}(\theta)\right|\:,
\end{equation*}
which in turn implies
\begin{eqnarray} \label{transference_pf_eq6}
&&\left\|\sup_{0\le t<2\pi}\left|e^{-it\psi_1\left(\sqrt{-\Delta}\right)} f_m   - e^{-it\psi_2\left(\sqrt{-\Delta}\right)} f_m \right|\right\|_{L^p(\X)} \nonumber\\
&\le & \sum_{j=1}^\infty \frac{1}{j!} \left\| \sup_{0\le t<2\pi} \left|e^{-it\psi_1\left(\sqrt{-\Delta}\right)} g_{j,m}\right|\right\|_{L^p(\X)}\:.
\end{eqnarray}
For any $\varepsilon>0$, set $\alpha_1:=\alpha_0+\frac{\varepsilon}{2}$ and $\alpha_2:=\alpha_0+\varepsilon$. Then applying the assumption (\ref{transference_pf_eq1}) and Lemma \ref{sobolev_comparison}, we get that
\begin{equation*}
\left\| \sup_{0\le t<2\pi} \left|e^{-it\psi_1\left(\sqrt{-\Delta}\right)} g_{j,m}\right|\right\|_{L^p(\X)} \lesssim \|g_{j,m}\|_{H^{\alpha_1}(\X)} \lesssim 2^{-\frac{m\varepsilon}{2}}  \|g_{j,m}\|_{H^{\alpha_2}(\X)}\:,
\end{equation*}
which upon plugging in (\ref{transference_pf_eq6}) yields,
\begin{equation} \label{transference_pf_eq7}
\left\|\sup_{0\le t<2\pi}\left|e^{-it\psi_1\left(\sqrt{-\Delta}\right)} f_m   - e^{-it\psi_2\left(\sqrt{-\Delta}\right)} f_m \right|\right\|_{L^p(\X)} \lesssim 2^{-\frac{m\varepsilon}{2}} \sum_{j=1}^\infty \frac{1}{j!} \|g_{j,m}\|_{H^{\alpha_2}(\X)}\:.
\end{equation}
Recalling the definition of $g_{j,m}$ along with (\ref{transference_pf_eq3}), we note that
\begin{equation*}
\|g_{j,m}\|_{H^{\alpha_2}(\X)} \le C^j\:\|f_m\|_{H^{\alpha_2}(\X)}\:,
\end{equation*}
which upon plugging in (\ref{transference_pf_eq7}) gives,
\begin{eqnarray} \label{transference_pf_eq8}
\left\|\sup_{0\le t<2\pi}\left|e^{-it\psi_1\left(\sqrt{-\Delta}\right)} f_m   - e^{-it\psi_2\left(\sqrt{-\Delta}\right)} f_m \right|\right\|_{L^p(\X)} &\lesssim & 2^{-\frac{m\varepsilon}{2}} \|f_m\|_{H^{\alpha_2}(\X)}\left(\sum_{j=1}^\infty \frac{C^j}{j!}\right)\nonumber\\
& \lesssim &  2^{-\frac{m\varepsilon}{2}} \|f_m\|_{H^{\alpha_2}(\X)}\:.
\end{eqnarray}
Hence for each $m \ge m_0+1$, applying (\ref{transference_pf_eq1}) to $f_m$ and combining it with (\ref{transference_pf_eq8}), we get that
\begin{eqnarray*}
&&\left\|\sup_{0 \le t < 2\pi} \left|e^{-it\psi_2\left(\sqrt{-\Delta}\right)} f_m \right|\right\|_{L^p(\X)} \\
&\le & \left\|\sup_{0\le t<2\pi}\left|e^{-it\psi_1\left(\sqrt{-\Delta}\right)} f_m   - e^{-it\psi_2\left(\sqrt{-\Delta}\right)} f_m \right|\right\|_{L^p(\X)} + \left\|\sup_{0 \le t < 2\pi} \left|e^{-it\psi_1\left(\sqrt{-\Delta}\right)} f_m \right|\right\|_{L^p(\X)} \\
& \lesssim &  2^{-\frac{m\varepsilon}{2}} \|f_m\|_{H^{\alpha_2}(\X)} + \|f_m\|_{H^{\alpha_1}(\X)}\:.
\end{eqnarray*}
Then a further application of Lemma \ref{sobolev_comparison} gives
\begin{equation*}
\left\|\sup_{0 \le t < 2\pi} \left|e^{-it\psi_2\left(\sqrt{-\Delta}\right)} f_m \right|\right\|_{L^p(\X)} \lesssim 2^{-\frac{m\varepsilon}{2}} \|f_m\|_{H^{\alpha_2}(\X)} \le 2^{-\frac{m\varepsilon}{2}} \|f\|_{H^{\alpha_2}(\X)}\:,
\end{equation*}
which upon plugging in (\ref{transference_pf_eq4}) and combined with (\ref{transference_pf_eq5}) yields
\begin{eqnarray*}
&&\left\|\sup_{0 \le t < 2\pi} \left|e^{-it\psi_2\left(\sqrt{-\Delta}\right)} f \right|\right\|_{L^p(\X)}\\
 & \le & \left\|\sup_{0 \le t < 2\pi} \left|e^{-it\psi_2\left(\sqrt{-\Delta}\right)} f_{m_0} \right|\right\|_{L^p(\X)} + \sum_{m=m_0+1}^\infty \left\|\sup_{0 \le t < 2\pi} \left|e^{-it\psi_2\left(\sqrt{-\Delta}\right)} f_m \right|\right\|_{L^p(\X)} \\
&\lesssim & \left(1+\sum_{m=m_0+1}^\infty 2^{-\frac{m\varepsilon}{2}} \right)  \|f\|_{H^{\alpha_2}(\X)} \\
&\lesssim &  \|f\|_{H^{\alpha_2}(\X)}\:.
\end{eqnarray*}
As $\varepsilon>0$ and $K$-biinvariant $f \in C^\infty(\X)$ were arbitrarily chosen, (\ref{transference_pf_eq2}) follows. This completes the proof of Theorem \ref{transference_principle}.
\end{proof}

\section{Sufficiency of $\alpha >1/3$ for $K$-biinvariant data}
In this section, we will prove Theorem \ref{result2}.

The unit circle $\T$ will be identified with the interval $[0,2\pi)$ equipped with the normalized Lebesgue measure. We first obtain the following variant of the Strichartz estimate on $\T$:
\begin{lemma} \label{Strichartz}
Let $N \ge 1$ be an integer and $\sa =\left\{a_n\right\}_{n=0}^{N-1} \subset \C$ be a sequence. Then for any $\varepsilon>0$,
\begin{equation*}
\left\|\sum_{n=0}^{N-1} a_n e^{-it\lambda^2_n}e^{\pm in\theta} \right\|_{L^6(\T \times \T)} \lesssim_\varepsilon N^\varepsilon \|\sa\|_{\ell^2}\:, 
\end{equation*}
where $\lambda^2_n$ is of the form appearing in Table \ref{table:2}.
\end{lemma}
\begin{proof}
By reflection $\theta \mapsto -\theta$, it is enough to consider $e^{-in\theta}$. As presented in (\ref{rank1_eigenvalue_form}), the general form of  $\lambda^2_n$ is given by,
\begin{equation*} 
\lambda^2_n =n(n+\sigma+\tau+1)\:,\:\:n \in \N \cup \{0\}\:,
\end{equation*}
where the values of $\sigma$ and $\tau$ are given in Table \ref{table:2}. Thus we note that while working in either of the cases, we have
\begin{equation*} 
\lambda^2_n =n(n+\ell)\:,\:\:n \in \N \cup \{0\}\:,
\end{equation*}
for some fixed non-negative integer $\ell$. Via Plancherel theorem, the $L^6$-norm appearing in the statement of Lemma \ref{Strichartz} can be written as,
\begin{equation*} 
\left\|\sum_{n=0}^{N-1} a_n e^{-itn(n+\ell)}e^{- in\theta} \right\|^6_{L^6(\T \times \T)} = C \sum_{u,v} \left|\sum_{\substack{n_1+n_2+n_3=u \\
n_1(n_1+\ell)+n_2(n_2+\ell)+n_3(n_3+\ell)=v}}a_{n_1}a_{n_2}a_{n_3}\right|^2\:,
\end{equation*}
for a positive constant $C$. We note that as $0 \le n_1,n_2,n_3 \le N-1$, we have that $u \lesssim N$ and $v \lesssim N^2$. Setting 
\begin{equation*}
r^{(\ell)}_{u,v} := \#\left\{(n_1,n_2,n_3) \in [0,N]^3 \mid \sum_{j=1}^3 n_j=u, \sum_{j=1}^3 n_j(n_j+\ell)=v \right\}\:,
\end{equation*}
we note by the Cauchy-Schwarz inequality that 
\begin{eqnarray*}
&&\sum_{u,v} \left|\sum_{\substack{n_1+n_2+n_3=u \\
n_1(n_1+\ell)+n_2(n_2+\ell)+n_3(n_3+\ell)=v}}a_{n_1}a_{n_2}a_{n_3}\right|^2\\ 
& \lesssim & \sum_{u,v} r^{(\ell)}_{u,v} \sum_{\substack{n_1+n_2+n_3=u \\
n_1(n_1+\ell)+n_2(n_2+\ell)+n_3(n_3+\ell)=v}}|a_{n_1}|^2\:|a_{n_2}|^2\:|a_{n_3}|^2 \\
& \le & \left(\max_{u,v} r^{(\ell)}_{u,v}\right) \|\sa\|^6_{\ell^2}\:.
\end{eqnarray*}
Thus it suffices to prove that for all $u,v$ as above and any $\varepsilon>0$,
\begin{equation*}
r^{(\ell)}_{u,v} \lesssim_\varepsilon N^\varepsilon\:.
\end{equation*}
Now for $\ell=0$, the proof of \cite[Proposition 2.36]{Bourgain-GAFA} yields
\begin{equation*}
r^{(0)}_{u,v} \lesssim_\varepsilon N^\varepsilon\:.
\end{equation*}
Then for $\ell \ge 1$, the same estimate follows from the simple observation that $r^{(\ell)}_{u,v}=r^{(0)}_{u,v-\ell u}$. This completes the proof of Lemma \ref{Strichartz}.
\end{proof}

As a consequence of Lemma \ref{Strichartz}, we obtain the following maximal estimate on $\T$:
\begin{lemma} \label{max_estimate_circle}
Let $N \ge 1$ be an integer and $\sa =\left\{a_n\right\}_{n=0}^{N-1} \subset \C$ be a sequence. Then for any $\varepsilon>0$,
\begin{equation*}
\left\|\sup_{t \in \T}\left|\sum_{n=0}^{N-1} a_n e^{-it\lambda^2_n}e^{\pm in\theta}\right| \right\|_{L^6(\T )} \lesssim_\varepsilon N^{\frac{1}{3}+\varepsilon} \|\sa\|_{\ell^2}\:, 
\end{equation*}
where $\lambda^2_n$ is of the form appearing in Table \ref{table:2}.
\end{lemma}
\begin{proof}
While working in either of the cases as classified in Table \ref{table:2}, we first consider the differential operator $D$ on $\T$ such that
\begin{equation*}
D(e^{in\theta}) = -\lambda^2_{|n|}\:e^{in\theta}\:, \: n \in \Z,\: \theta \in \T\:. 
\end{equation*}
Then we look at the solution of the corresponding Schr\"odinger equation:
\begin{equation} \label{max_estimate_circle_eq1}
\begin{cases}
	 i\frac{\partial u}{\partial t} +Du=0\:,\:\:\:  &(\theta,t) \in \T \times \T\:, \\
	u(0,\cdot)=f\:, &\text{ on } \T \:,
	\end{cases}
\end{equation}
with initial data
\begin{equation*}
f(\theta):=\sum_{n=0}^{N-1}a_n e^{\pm in\theta}\:,\:\:\theta \in \T\:.
\end{equation*}
By an application of the Fundamental theorem of Calculus in time $t$, we have for any $(\theta,t) \in \T \times \T$,
\begin{equation*}
|u(\theta,t)|^6 = |f(\theta)|^6 \:+\: 3\int_0^t |u(\theta,s)|^4\: 2\:\text{Re}\left(u(\theta,s)\overline{u_s(\theta,s)}\right)\:ds\:.
\end{equation*}
Thus taking the supremum in $t$ and then integrating in $\theta$, it follows that
\begin{equation} \label{max_estimate_circle_eq2}
\left\|\sup_{t \in \T}\left|\sum_{n=0}^{N-1} a_n e^{-it\lambda^2_n}e^{\pm in\theta}\right| \right\|^6_{L^6(\T )} \le \|f\|^6_{L^6(\T )} \:+\: 6\int_\T\int_\T |u(\theta,s)|^5\:|u_s(\theta,s)|\:ds\:d\theta.
\end{equation}
Now by the Sobolev embedding on $\T$, we have
\begin{equation} \label{max_estimate_circle_eq3}
\|f\|_{L^6(\T )} \lesssim N^{\frac{1}{3}}\:\|f\|_{L^2(\T )}=N^{\frac{1}{3}}\:\|\sa\|_{\ell^2}\:.
\end{equation}
For the other term, we apply H\"older's inequality to get
\begin{equation} \label{max_estimate_circle_eq4}
\int_\T\int_\T |u(\theta,s)|^5\:|u_s(\theta,s)|\:ds\:d\theta \lesssim \|u\|^5_{L^6(\T \times \T)}\:\|u_s\|_{L^6(\T \times \T)}\:.
\end{equation}
Then by Lemma \ref{Strichartz}, we get that for any $\varepsilon>0$,
\begin{equation} \label{max_estimate_circle_eq5}
\|u\|_{L^6(\T \times \T)} = \left\|\sum_{n=0}^{N-1} a_n e^{-it\lambda^2_n}e^{\pm in\theta} \right\|_{L^6(\T \times \T)} \lesssim_\varepsilon N^\varepsilon \|\sa\|_{\ell^2}\:.
\end{equation}
To estimate $\|u_t\|_{L^6(\T \times \T)}$, we note that
\begin{equation*}
u_t= ie^{itD} Df\:,
\end{equation*}
where $e^{itD} Df$ is the solution of the Schr\"odinger equation (\ref{max_estimate_circle_eq1}) with initial data $Df$. Then applying Lemma \ref{Strichartz} to $Df$ and recalling Remark \ref{rank1_eigenvalue_remark}, we obtain for any $\varepsilon>0$,
\begin{equation} \label{max_estimate_circle_eq6}
\|u_t\|_{L^6(\T \times \T)} \lesssim_\varepsilon N^{2+\varepsilon}\|\sa\|_{\ell^2} \:.
\end{equation}
Thus combining (\ref{max_estimate_circle_eq2})-(\ref{max_estimate_circle_eq6}), the result follows.
\end{proof}

\begin{remark} \label{Strichartz_remark}
In retrospect, Lemma \ref{Strichartz} can be viewed as a Strichartz estimate for even initial data on $\T$ corresponding to the Schr\"odinger operator (\ref{max_estimate_circle_eq1}). 
\end{remark}

We will also require the following Sobolev embedding for $K$-biinvariant functions:
\begin{lemma} \label{sobolev_embedding}
Let $\X$ be a rank one Riemannian symmetric space of compact type and dimension $d \ge 1$. Then for all $K$-biinvariant functions $f$ and all $\varepsilon>0$, 
\begin{itemize}
\item[(i)] near the origin, we have
\begin{equation*}
\|f\|_{L^\infty\left(\left[0,\frac{\pi}{2}\right]\right)} \lesssim_\varepsilon \|f\|_{H^{\frac{d}{2}+\varepsilon}(\X)}\:;
\end{equation*}
\item[(ii)] near the antipodal manifold, we have
\begin{equation*}
\|f\|_{L^\infty\left(\left[\frac{\pi}{2},\pi\right]\right)} \lesssim_\varepsilon \|f\|_{H^{\alpha_0+\varepsilon}(\X)}\:,
\end{equation*}
where 
\begin{equation*}
\alpha_0 =
\begin{cases}
\frac{d}{2}\:\:\:&\text{if } \X=\Sb^d\:,\\
\frac{1}{2}\:\:\:&\text{if } \X=P^d(\R)\:,\\
1\:\:\:&\text{if } \X=P^d(\C)\:,\\
2\:\:\:&\text{if } \X=P^d(\Hb)\:,\\
4\:\:\:&\text{if } \X=P^{16}(Cay)\:.
\end{cases}
\end{equation*}
\end{itemize}
\end{lemma}

\begin{proof}
To prove part (i), by density, it suffices to prove the inequality for $K$-biinvariant functions $f \in C^\infty(\X)$ with $Supp(f) \subset [0,\frac{\pi}{2}]$. For such a function $f$, its spherical Fourier series (as in (\ref{spherical_fourier_series})) is of the form
\begin{equation*}
f(\theta)=\sum_{n=0}^\infty a_n \:\tilde{Z}_n(\theta)\:,\:\:\theta \in [0,\pi]\:,
\end{equation*}
with $a_n \in \C$. Performing the dyadic decomposition,
\begin{equation*}
f=\sum_{m=0}^\infty \sum_{n=2^m-1}^{2^{m+1}-2} a_n\:\tilde{Z}_n\:,
\end{equation*} 
we note that for any $\alpha \ge 0$
\begin{equation*}
\|f\|^2_{H^{\alpha}(\X)} \asymp \sum_{m=0}^\infty 2^{2\alpha m} \sum_{n=2^m-1}^{2^{m+1}-2} |a_n|^2\:.
\end{equation*}
Thus to prove part (i), it suffices to show that for any $\varepsilon>0$,
\begin{equation} \label{sobolev_embedding_pf_eq1}
\|g_N\|_{L^\infty\left(\left[0,\frac{\pi}{2}\right]\right)} \lesssim N^{\frac{d}{2}+\varepsilon}\|g_N\|_{L^2(\X)}\:,
\end{equation}
where 
\begin{equation*}
g_N:= \sum_{n=N-1}^{2(N-1)} a_n\:\tilde{Z}_n\:,\:N=2^m\:,\:m \in \N \cup \{0\}\:,
\end{equation*}
as then by (\ref{sobolev_embedding_pf_eq1}) and the Cauchy-Schwarz inequality,
\begin{eqnarray*}
\|f\|_{L^\infty\left(\left[0,\frac{\pi}{2}\right]\right)} \le \sum_{m=0}^\infty \|g_{2^m}\|_{L^\infty\left(\left[0,\frac{\pi}{2}\right]\right)} &\lesssim & \sum_{m=0}^\infty {(2^m)}^{\frac{d}{2}+\varepsilon}  \|g_{2^m}\|_{L^2(\X)} \\
&=& \sum_{m=0}^\infty 2^{-m\varepsilon}\:{(2^m)}^{\frac{d}{2}+2\varepsilon}  \|g_{2^m}\|_{L^2(\X)} \\
& \le & \left(\sum_{m=0}^\infty 2^{-2m\varepsilon}\right)^{\frac{1}{2}} \left(\sum_{m=0}^\infty 2^{2\left(\frac{d}{2}+2\varepsilon\right)m}\|g_{2^m}\|^2_{L^2(\X)}\right)^{\frac{1}{2}} \\
&\lesssim_{\varepsilon} & \|f\|_{H^{\frac{d}{2}+2\varepsilon}(\X)}\:.
\end{eqnarray*}
The result would then follow as $\varepsilon>0$ is arbitrary.

Now (\ref{sobolev_embedding_pf_eq1}) follows from the $L^\infty$ bound in part (i) of Lemma \ref{Linfty_bound} and the Cauchy-Schwarz inequality, as for all $\theta \in [0,\pi/2]$,
\begin{eqnarray*}
\left|g_N(\theta)\right| = \left|\sum_{n=N-1}^{2(N-1)} a_n\:\tilde{Z}_n(\theta)\right| &\le & \sum_{n=N-1}^{2(N-1)} \left|a_n\right|\:\left\|\tilde{Z}_n\right\|_{L^\infty\left(\left[0,\frac{\pi}{2}\right]\right)} \\
& \lesssim & \sum_{n=N-1}^{2(N-1)} \left|a_n\right|\: {\left(1+n\right)}^{\frac{d-1}{2}} \\
& \le & \left(\sum_{n=N-1}^{2(N-1)}{\left(1+n\right)}^{d-1}\right)^{\frac{1}{2}}\|g_N\|_{L^2(\X)} \\
& \lesssim & N^{\frac{d}{2}}\:\|g_N\|_{L^2(\X)}\:.
\end{eqnarray*}
This completes the proof of  part (i) Lemma \ref{sobolev_embedding}.

Part (ii) also follows similarly by applying part (ii) of Lemma \ref{Linfty_bound}.
\end{proof}

We now present the proof of Theorem \ref{result2}.
\begin{proof}[Proof of Theorem \ref{result2}]
The proof is divided into two steps.

{\bf Step 1:} We first focus on the case of the Schr\"odinger equation. It suffices to prove the maximal estimate
\begin{equation} \label{result2_pf_eq1}
\left\|\sup_{0 \le t < 2\pi} \left|e^{it\Delta} f \right|\right\|_{L^1(\X)} \lesssim \|f\|_{H^\alpha(\X)}\:,
\end{equation}
for all $\alpha >1/3$ and all $K$-biinvariant $f \in C^\infty(\X)$.

The maximal estimate (\ref{result2_pf_eq1}) for $\X=\Sb^d$ is proved in \cite[Theorem 1.3]{CDLY}. We now recall that the real projective space $P^d(\R)$ can be obtained from $\Sb^d$ by identifying the antipodal points:
\begin{eqnarray*}
\pi : &&\Sb^d\:\: \to \:\: P^d(\R) \\
&& \pm x \:\:\mapsto \:\: [x]\:,
\end{eqnarray*}
with the projection map $\pi$ being a local isometry.  Thus the functions on $P^d(\R)$ can be identified as even functions on $\Sb^d$ and if $f_e$ is the even function on $\Sb^d$ corresponding to the function $f$ on $P^d(\R)$, then $\Delta_{P^d(\R)} f=\Delta_{\Sb^d}f_e$. So the corresponding Schr\"odinger propagations can also be similarly identified. Thus the maximal estimate (\ref{result2_pf_eq1}) for $\X=P^d(\R)$ follows from the corresponding result on $\Sb^d$. So we will give the proof for the complementary cases: $\X=P^d(\C),\:P^d(\Hb)$ and $P^{16}(Cay)$.

For a $K$-biinvariant $f \in C^\infty(\X)$, its spherical Fourier series is of the form
\begin{equation*}
f(\theta)=\sum_{n=0}^\infty a_n \:\tilde{Z}_n(\theta)\:,\:\:\theta \in [0,\pi]\:,
\end{equation*}
with $a_n \in \C$. Then its Schr\"odinger propagation is also $K$-biinvariant and is given by,
\begin{equation*}
e^{it\Delta} f(\theta)=\sum_{n=0}^\infty e^{-it\lambda^2_n}\; a_n \:\tilde{Z}_n(\theta)\:,\:\:\theta \in [0,\pi]\:,
\end{equation*}
where $\lambda^2_n$ are as classified in Table \ref{table:2}. For notational convenience, we write the maximal function as,
\begin{equation*}
S^*f(\theta):= \sup_{0 \le t < 2\pi} \left|e^{it\Delta} f (\theta)\right|\:,\:\:\theta \in [0,\pi]\:.
\end{equation*}
Then in view of (\ref{density}), we get that
\begin{equation*}
\left\|\sup_{0 \le t < 2\pi} \left|e^{it\Delta} f \right|\right\|_{L^1(\X)} \asymp \int_0^\pi S^*f(\theta)\:{\left(\sin \frac{\theta}{2}\right)}^{M_1} {(\sin \theta)}^{M_2}\:d\theta\:,
\end{equation*}
where $(M_1,M_2)$ are as in Table \ref{table:1}. Now dyadically decomposing $f$ as in the proof of Lemma \ref{sobolev_embedding}, we note that to obtain (\ref{result2_pf_eq1}), it suffices to prove that for any $\varepsilon>0$,
\begin{equation} \label{result2_pf_eq2}
\int_0^\pi S^*g_N(\theta)\:{\left(\sin \frac{\theta}{2}\right)}^{M_1} {(\sin \theta)}^{M_2}\:d\theta \lesssim_\varepsilon N^{\frac{1}{3}+\varepsilon}\:\|\sa_N\|_{\ell^2}\:,
\end{equation}
where 
\begin{equation*}
g_N:= \sum_{n=N-1}^{2(N-1)} a_n\:\tilde{Z}_n\:\:\text{   and    } \sa_N := \{a_n\}_{n=N-1}^{2(N-1)}\:\:,\:N=2^m\:,\:m \in \N \cup \{0\}\:.
\end{equation*} 
Now proceeding as in the proof of the Sobolev embedding (part (i) of Lemma \ref{sobolev_embedding}), we obtain
\begin{equation} \label{result2_pf_eq3}
\|S^*g_N\|_{L^\infty\left(\left[0,\frac{\pi}{2}\right]\right)} \lesssim N^{\frac{d}{2}}\|\sa_N\|_{\ell^2}\:.
\end{equation}
We fix a small $c>0$ and use the above estimate for  $\theta$ near $0$. Indeed, using (\ref{result2_pf_eq3}), the facts that  $\sin \theta \le \theta$ near $0$, $M_1+M_2+1=d$ and then performing an elementary integration, we obtain
\begin{equation}\label{result2_pf_eq4}
\int_0^{cN^{-1}} S^*g_N(\theta)\:{\left(\sin \frac{\theta}{2}\right)}^{M_1} {(\sin \theta)}^{M_2}\:d\theta \lesssim N^{-\frac{d}{2}} \|\sa_N\|_{\ell^2}\:.
\end{equation}
Similarly, using part (ii) of Lemma \ref{sobolev_embedding} and the fact that $\sin\frac{\theta}{2} \le 1$, we obtain near $\theta=\pi$,
\begin{equation*}
\int_{\pi-cN^{-1}}^\pi S^*g_N(\theta)\:{\left(\sin \frac{\theta}{2}\right)}^{M_1} {(\sin \theta)}^{M_2}\:d\theta \lesssim N^{\alpha_0} \|\sa_N\|_{\ell^2}\int_{\pi-cN^{-1}}^\pi {(\sin \theta)}^{M_2}\:d\theta  \:.
\end{equation*}
Then performing the change of variable $\xi:=\pi - \theta$, we see that
\begin{equation*}
\int_{\pi-cN^{-1}}^\pi {(\sin \theta)}^{M_2}\:d\theta = \int_0^{cN^{-1}} {(\sin \xi)}^{M_2}\:d\xi\:. 
\end{equation*}
Now using the fact that $\sin \xi \le \xi$ and then performing an elementary integration, we obtain
\begin{equation*}
\int_0^{cN^{-1}} {(\sin \xi)}^{M_2}\:d\xi \lesssim N^{-(M_2+1)}\:,
\end{equation*}
and hence we get for $\theta$ near $\pi$,
\begin{equation}\label{result2_pf_eq5}
\int_{\pi-cN^{-1}}^\pi S^*g_N(\theta)\:{\left(\sin \frac{\theta}{2}\right)}^{M_1} {(\sin \theta)}^{M_2}\:d\theta \lesssim N^{\alpha_0-(M_2+1)}\:\|\sa_N\|_{\ell^2}\:.
\end{equation}
At this juncture, we note that for the cases under consideration, we have 
\begin{equation} \label{result2_pf_eq6}
\alpha_0-(M_2+1) = 
\begin{cases}
-1\:\:\:&\text{if } \X=P^d(\C)\:,\\
-2\:\:\:&\text{if } \X=P^d(\Hb)\:,\\
-4\:\:\:&\text{if } \X=P^{16}(Cay)\:.
\end{cases}
\end{equation}
Next, in the region $cN^{-1} \le \theta \le \pi - cN^{-1}$, invoking the expansion of Jacobi polynomials given by Lemma \ref{jacobi_polynomial} and  the estimate (\ref{jacobi_poly_L2}), we have the uniform estimate
\begin{eqnarray} \label{result2_pf_eq7}
\tilde{Z}_n(\theta)&=&\frac{c_n}{\left(\sin\frac{\theta}{2}\right)^{\sigma+\frac{1}{2}}\left(\cos\frac{\theta}{2}\right)^{\tau+\frac{1}{2}}}\cos\left(\left(n+\frac{\sigma+\tau+1}{2}\right)\theta-\frac{\left(\sigma +\frac{1}{2}\right)\pi}{2}\right)\nonumber\\
&&+\frac{\mathcal{O}(1)}{n\sin\theta \left(\sin\frac{\theta}{2}\right)^{\sigma+\frac{1}{2}}\left(\cos\frac{\theta}{2}\right)^{\tau+\frac{1}{2}}}\:,
\end{eqnarray}
where $c_n$ are positive constants, bounded above and below and the parameters $(\sigma,\tau)$ are given in Table \ref{table:2}. Invoking the decomposition (\ref{result2_pf_eq7}) and expanding cosine in terms of exponentials, in the region $cN^{-1} \le \theta \le \pi - cN^{-1}$, we  get
\begin{equation*}
S^*g_N(\theta) \lesssim S^*_1g_N(\theta) \:+\: S^*_2g_N(\theta)\:,
\end{equation*} 
where 
\begin{eqnarray*}
&&S^*_1g_N(\theta):= \frac{1}{\left(\sin\frac{\theta}{2}\right)^{\sigma+\frac{1}{2}}\left(\cos\frac{\theta}{2}\right)^{\tau+\frac{1}{2}}}\sum_{\pm}\sup_{0\le t<2\pi} \left|\sum_{n=N-1}^{2(N-1)} a_nc_n\:e^{-it\lambda^2_n}\:e^{\pm in\theta} \right| \:,\\
&&S^*_2g_N(\theta):= \frac{1}{\sin\theta \left(\sin\frac{\theta}{2}\right)^{\sigma+\frac{1}{2}}\left(\cos\frac{\theta}{2}\right)^{\tau+\frac{1}{2}}}\sup_{0\le t<2\pi}\left|\sum_{n=N-1}^{2(N-1)}\frac{a_n\:e^{-it\lambda^2_n} \mathcal{O}(1)}{n}\right|\:.
\end{eqnarray*}
We first estimate $S^*_2g_N$. By an application of the Cauchy-Schwarz inequality in frequency $n$, we get that for any $\theta \in \left[cN^{-1}, \pi - cN^{-1}\right]$,
\begin{equation*}
S^*_2g_N(\theta) \lesssim \frac{\|\sa_N\|_{\ell^2}}{\sqrt{N}\sin\theta \left(\sin\frac{\theta}{2}\right)^{\sigma+\frac{1}{2}}\left(\cos\frac{\theta}{2}\right)^{\tau+\frac{1}{2}}}\:,
\end{equation*}
which yields
\begin{eqnarray*}
&&\int_{cN^{-1}}^{\pi-cN^{-1}} S^*_2g_N(\theta)\:{\left(\sin \frac{\theta}{2}\right)}^{M_1} {(\sin \theta)}^{M_2}\:d\theta \\
&\lesssim & \frac{\|\sa_N\|_{\ell^2}}{\sqrt{N}}\int_{cN^{-1}}^{\pi-cN^{-1}} \frac{{\left(\sin \frac{\theta}{2}\right)}^{M_1} {(\sin \theta)}^{M_2}}{\sin\theta \left(\sin\frac{\theta}{2}\right)^{\sigma+\frac{1}{2}}\left(\cos\frac{\theta}{2}\right)^{\tau+\frac{1}{2}}}\:d\theta\:.
\end{eqnarray*}
Now by elementary estimates, keeping in mind the values of the parameters $M_1,M_2,\sigma,\tau$, we get that
\begin{equation} \label{result2_pf_eq8}
\int_{cN^{-1}}^{\pi-cN^{-1}} S^*_2g_N(\theta)\:{\left(\sin \frac{\theta}{2}\right)}^{M_1} {(\sin \theta)}^{M_2}\:d\theta \lesssim N^{\alpha_1}\:\|\sa_N\|_{\ell^2}\:,
\end{equation}
where 
\begin{equation} \label{result2_pf_eq9}
\alpha_1 = \begin{cases}
0\:\:\:&\text{ if } \X=P^d(\C)\:,\\
-\frac{1}{2}\:\:\:&\text{ if } \X=P^d(\Hb),\:P^{16}(Cay)\:.
\end{cases}
\end{equation}
Finally, by using the values of the parameters $M_1,M_2,\sigma,\tau$, Lemma \ref{max_estimate_circle} and the fact that $c_n$ are positive constants, bounded above and below, we get that for any $\varepsilon>0$,
\begin{eqnarray} \label{result2_pf_eq11}
&&\int_{cN^{-1}}^{\pi-cN^{-1}} S^*_1g_N(\theta)\:{\left(\sin \frac{\theta}{2}\right)}^{M_1} {(\sin \theta)}^{M_2}\:d\theta \nonumber\\
& = & \int_{cN^{-1}}^{\pi-cN^{-1}} \frac{{\left(\sin \frac{\theta}{2}\right)}^{M_1} {(\sin \theta)}^{M_2}}{\left(\sin\frac{\theta}{2}\right)^{\sigma+\frac{1}{2}}\left(\cos\frac{\theta}{2}\right)^{\tau+\frac{1}{2}}}\left(\sum_{\pm}\sup_{0\le t<2\pi} \left|\sum_{n=N-1}^{2(N-1)} a_nc_n\:e^{-it\lambda^2_n}\:e^{\pm in\theta} \right|\right)\;d\theta \nonumber\\
&\lesssim & \left\|\sum_{\pm}\sup_{0\le t<2\pi} \left|\sum_{n=N-1}^{2(N-1)} a_nc_n\:e^{-it\lambda^2_n}\:e^{\pm in\theta} \right|\right\|_{L^1(\T)} \nonumber\\
&\lesssim_\varepsilon & N^{\frac{1}{3}+\varepsilon}\:\|\sa_N\|_{\ell^2}\:.
\end{eqnarray}
Thus combining (\ref{result2_pf_eq4})-(\ref{result2_pf_eq6}) and (\ref{result2_pf_eq8})-(\ref{result2_pf_eq11}), we get (\ref{result2_pf_eq2}), which completes the proof of Theorem \ref{result2} for the Schr\"odinger equation.

\medskip

{\bf Step 2:} By the transference principle (Theorem \ref{transference_principle}), the results in Theorem \ref{result2} for the Boussinesq equation and the Beam equation now follow from the result of the Schr\"odinger equation proved in Step $1$, in view of Remark \ref{examples_remark}. This completes the proof of Theorem \ref{result2}.
\end{proof}

\section{Failure of $\alpha <1/4$}
In this section, we prove Theorem \ref{result3}. But we first see the following counting result, whose proof uses asymptotics of arithmetic functions:
\begin{lemma} \label{counting_lemma}
Let $N \in \N$ be large and $\varepsilon>0$ be small. Then there exists $C>0$ independent of $N$ such that the set
\begin{equation} \label{set}
E_N := \bigcup_{\substack{q \text{ odd } :\: \sqrt{N} \le q \le 2\sqrt{N} \\ p \text{ even } :\:2\varepsilon < \frac{2\pi p}{q}<\pi - 2\varepsilon}} \left(\frac{2\pi p}{q}+\frac{\pi}{16N}\:,\:\frac{2\pi p}{q}+\frac{\pi}{8N}\right)\:,
\end{equation}
has Lebesgue measure $|E_N|\ge C$\:.
\end{lemma}
\begin{proof}
We first note that the intervals considered in the above union are either disjoint or identical. Indeed, let $(p_1,q_1)$ and $(p_2,q_2)$ be two pairs as in the definition (\ref{set}) such that $$\frac{p_1}{q_1} \le \frac{p_2}{q_2}\:.$$ Now if the corresponding intervals have non-empty intersection, then we must have
\begin{equation*}
\frac{2\pi p_2}{q_2}+\frac{\pi}{16N} \le \frac{2\pi p_1}{q_1}+\frac{\pi}{8N}\:,
\end{equation*}
which yields
\begin{equation*}
p_2q_1-p_1q_2 \le \frac{q_1q_2}{32N} \le \frac{1}{8}\:.
\end{equation*}
Now since $p_2q_1-p_1q_2$ is a non-negative integer, we must have $p_2q_1-p_1q_2=0$, that is,
\begin{equation*}
\frac{p_1}{q_1} = \frac{p_2}{q_2}\:,
\end{equation*}
and hence the corresponding intervals are identical. Thus we get that
\begin{equation} \label{counting_lemma_eq1}
|E_N| \ge (\text{the number of intervals appearing in the definition of } E_N) \times \frac{\pi}{16N}\:.
\end{equation}
We now consider the pairs $(p,q)$ such that $p=2p'$ with $\gcd(p',q)=1$. Then the condition on $p$ in the definition of $E_N$ becomes
\begin{equation*}
2\varepsilon < \frac{4\pi p'}{q}<\pi - 2\varepsilon\:,
\end{equation*}
that is,
\begin{equation*}
\frac{\varepsilon q}{2\pi} < p' < \frac{q}{4}-\frac{\varepsilon q}{2\pi}\:.
\end{equation*}
Now as
\begin{equation*}
\left|\left(\frac{\varepsilon q}{2\pi}, \frac{q}{4}-\frac{\varepsilon q}{2\pi}\right)\right|=
 \frac{q}{4} - \frac{\varepsilon q}{\pi} > \frac{q}{5}\:,\:\:\text{  for }\varepsilon < \frac{\pi}{20}\:\:,
\end{equation*}
for each odd $q$ with $\sqrt{N} \le q \le 2\sqrt{N}$, we look at the cardinality of the set $\mathcal{I}_q$ defined as,
\begin{equation*}
\mathcal{I}_q := \left\{j \in \N \mid j \le \frac{q}{5}\text{ and } \gcd(j,q)=1\right\}\:.
\end{equation*}
To proceed, we use the fundamental identity of the M\"obius function given by (\ref{mobius_identity}) to write,
\begin{eqnarray*}
\#\mathcal{I}_q = \sum_{j \le \frac{q}{5}}\:\: \sum_{m| \gcd(j,q)} \mu(m) = \sum_{m|q} \mu(m) \:\:\sum_{\substack{j \le \frac{q}{5} \\ m|j}} 1 &=& \sum_{m|q} \mu(m) \:\: \left\lfloor \frac{q}{5m}\right\rfloor \\
&=& \sum_{m|q} \mu(m) \:\: \left( \frac{q}{5m} \:+\: \mathcal{O}(1)\right) \\
&=& \frac{q}{5}\sum_{m|q} \frac{\mu(m)}{m}  \:+\: \mathcal{O}\left(d(q)\right)\:,
\end{eqnarray*} 
where $d(q)$ is the divisor function. Now by M\"obius inversion (\ref{mobius_inversion}), we have
\begin{equation*}
\#\mathcal{I}_q= \frac{\varphi(q)}{5} \:+\:\mathcal{O}(d(q))\:,
\end{equation*}
which upon invoking the asymptotics of Euler's totient function $\varphi$ and the divisor function $d$ given by (\ref{asymptotics}), yields
\begin{equation*}
\#\mathcal{I}_q \ge \frac{\varphi(q)}{6}\:,
\end{equation*}
for $q$ sufficiently large. This is achieved by taking $N$ sufficiently large. Now as $p=2p'$ with $\gcd(p',q)=1$ and $q$ is odd, we have that $\gcd(p,q)=1$ and hence each of the fractions $\frac{2\pi p}{q}$ are unique. Thus it follows that,
\begin{eqnarray} \label{counting_lemma_eq2}
\text{the number of intervals appearing in the definition of } E_N &\ge & \sum_{\substack{\sqrt{N} \le q \le 2\sqrt{N} \\ q \text{ odd }}} \#\mathcal{I}_q \nonumber\\
&\ge & \frac{1}{6} \sum_{\substack{\sqrt{N} \le q \le 2\sqrt{N} \\ q \text{ odd }}} \varphi(q)\:.
\end{eqnarray}
This leads us to the following estimate, which by the M\"obius inversion (\ref{mobius_inversion}),
\begin{eqnarray*}
\sum_{\substack{j \le n \\ j \text{ odd }}} \varphi(j) = \sum_{\substack{j \le n \\ j \text{ odd }}}\sum_{m|j}  j\: \frac{\mu(m)}{m} &=& \sum_{\substack{m \le n \\ m \text{ odd }}} \mu(m) \sum_{\substack{\ell \le \frac{n}{m} \\ \ell \text{ odd }}} \ell \\
&=& \sum_{\substack{m \le n \\ m \text{ odd }}} \mu(m) {\left\lfloor \frac{\left(\frac{n}{m}+1\right)}{2} \right\rfloor}^2 \\
&=& \sum_{\substack{m \le n \\ m \text{ odd }}} \mu(m) \left\{ \frac{\left(\frac{n}{m}+1\right)}{2} + \mathcal{O}(1) \right\}^2 \\
&=& \frac{n^2}{4} \sum_{\substack{m \le n \\ m \text{ odd }}} \frac{\mu(m)}{m^2} + \mathcal{O}\left(n \sum_{m \le n}\frac{1}{m}\right) + \mathcal{O}(n) \\
&=& \frac{n^2}{4} \sum_{\substack{m \le n \\ m \text{ odd }}} \frac{\mu(m)}{m^2} + \mathcal{O}(n \log n) \\
&=& (1+o(1)) \frac{\zeta}{4}n^2\:, \text{ as } n \to \infty \:,
\end{eqnarray*}
where $\zeta:=\displaystyle\sum_{ m \text{ odd }} \frac{\mu(m)}{m^2}$, as the series converges absolutely. Hence for $N$ sufficiently large,
\begin{equation} \label{counting_lemma_eq3}
\sum_{\substack{\sqrt{N} \le q \le 2\sqrt{N} \\ q \text{ odd }}} \varphi(q) = \sum_{\substack{q \le 2\sqrt{N} \\ q \text{ odd }}} \varphi(q) - \sum_{\substack{q \le \sqrt{N} \\ q \text{ odd }}} \varphi(q) = (1+o(1)) \frac{3\zeta}{4}N\:.
\end{equation}
Thus combining (\ref{counting_lemma_eq1})-(\ref{counting_lemma_eq3}), we obtain
\begin{equation*}
|E_N| > \frac{\pi\zeta}{256}\:.
\end{equation*}
This completes the proof of Lemma \ref{counting_lemma}.
\end{proof}

We now present the proof of Theorem \ref{result3}.
\begin{proof}[Proof of Theorem \ref{result3}]
The proof is divided into two steps.

{\bf Step 1:} We first focus on the case of the Schr\"odinger equation. The failure of the maximal estimate for $\X=\Sb^d$ is proved in \cite[Theorem 1.3]{CDLY}. So we will give the proof for the complementary cases: $\X=P^d(\R)$ ($d$ odd), $P^d(\C),\:P^d(\Hb)$ and $P^{16}(Cay)$.  

Let $N \in \N$ be large and we consider the $K$-biinvariant function,
\begin{equation*}
f_N:= \sum_{n=0}^{N-1} \frac{\tilde{Z}_n}{c_n}\:,
\end{equation*}
where $c_n$ are as in (\ref{result2_pf_eq7}), with $c_0:=1$. Now as for any $\alpha \ge 0$,
\begin{equation*}
\|f_N\|_{H^\alpha} \lesssim N^{\alpha +\frac{1}{2}}\:,
\end{equation*}
it suffices to prove that
\begin{equation} \label{result3_pf_eq1}
\sup_{0\le t <2\pi}\left| e^{it\Delta} f_N(\theta)\right| \gtrsim N^{3/4}\:\:\:\text{ on } E_N\:,
\end{equation}
where the implicit constant above is independent of the choice of $N$ and $E_N$ is the set defined in Lemma \ref{counting_lemma}, as then by Lemma \ref{counting_lemma},
\begin{equation*}
N^{3/4} \lesssim \left\|\sup_{0\le t <2\pi} \left| e^{it\Delta} f_N\right|\right\|_{L^1(\X)} \lesssim \|f_N\|_{H^\alpha} \lesssim N^{\alpha +\frac{1}{2}}\:,
\end{equation*}
which forces $\alpha \ge 1/4$ as $N \to \infty$.

To prove (\ref{result3_pf_eq1}), we fix a small $\varepsilon>0$ and by invoking the expansion of Jacobi polynomials (\ref{result2_pf_eq7}), we have the uniform estimate
\begin{equation*}
\frac{\tilde{Z}_n(\theta)}{c_n}=\frac{1}{\left(\sin\frac{\theta}{2}\right)^{\sigma+\frac{1}{2}}\left(\cos\frac{\theta}{2}\right)^{\tau+\frac{1}{2}}}\cos\left(\left(n+\frac{\sigma+\tau+1}{2}\right)\theta-\frac{\left(\sigma +\frac{1}{2}\right)\pi}{2}\right)+\frac{\mathcal{O}(1)}{n+1}\:,
\end{equation*}
for $\theta \in [\varepsilon,\pi-\varepsilon]$, where the parameters $\sigma, \tau$ are given in Table \ref{table:2}.

Thus for $\theta \in [\varepsilon,\pi-\varepsilon]$, the Schr\"odinger propagation of $f_N$ is given by,
\begin{eqnarray} \label{result3_pf_eq2}
e^{it\Delta} f_N(\theta)&=&\frac{1}{\left(\sin\frac{\theta}{2}\right)^{\sigma+\frac{1}{2}}\left(\cos\frac{\theta}{2}\right)^{\tau+\frac{1}{2}}} \sum_{n=0}^{N-1} e^{-itn(n+\sigma+\tau+1)}\cos\left(n\theta+\kappa_{\sigma,\tau}(\theta)\right)\nonumber\\
&&+\mathcal{O}(\log N)\:,
\end{eqnarray}
where
\begin{equation*} 
\kappa_{\sigma,\tau}(\theta):= \frac{1}{2}(\sigma + \tau +1)\theta - \frac{1}{2}\left(\sigma + \frac{1}{2}\right)\pi \:.
\end{equation*}
We set,
\begin{equation*}
S(t,\theta):= \sum_{n=0}^{N-1} e^{-itn(n+\sigma+\tau+1)}\cos\left(n\theta+\kappa_{\sigma,\tau}(\theta)\right)\:.
\end{equation*}
Now expanding the cosine in terms of exponentials, we write
\begin{equation} \label{St}
S(t,\theta)= \frac{1}{2} \sum_{\pm} S_{\pm}(t,\theta)\:e^{\pm i \kappa_{\sigma ,\tau}(\theta)}\:,
\end{equation}
where 
\begin{equation*}
S_{\pm}(t,\theta) := \sum_{n=0}^{N-1} e^{-itn(n+\sigma+\tau+1)} \:e^{\pm in\theta}\:.
\end{equation*}
We now gradually obtain the divergence on the set $E_N$. First, we consider $t=\frac{2\pi}{q}$ such that $q$ is an odd integer with $\sqrt{N} \le q \le 2\sqrt{N}$. Then we have,
\begin{equation*}
S_+\left(\frac{2\pi}{q},\theta\right)= \sum_{n=0}^{N-1} e^{-2\pi i \frac{n(n+\sigma+\tau+1)}{q}} \:e^{in\theta}\:.
\end{equation*}
We note that in our cases, $\sigma + \tau +1 \in \N$ and hence the function $n \mapsto e^{-2\pi i \frac{n(n+\sigma+\tau+1)}{q}}$ is $q$-periodic. This enables us to decompose $S_+\left(\frac{2\pi}{q},\theta\right)$ as follows,
\begin{eqnarray} \label{result3_pf_eq3} 
S_+\left(\frac{2\pi}{q},\theta\right)&=& \left(\sum_{n=0}^{q\left\lfloor\frac{N}{q}\right\rfloor - 1} \:+\: \sum_{n=q\left\lfloor\frac{N}{q}\right\rfloor}^{N-1}\right)e^{-2\pi i \frac{n(n+\sigma+\tau+1)}{q}} \:e^{in\theta}\nonumber\\
&=& \left(\sum_{\ell=0}^{\left\lfloor\frac{N}{q}\right\rfloor - 1} e^{i\ell q \theta}\right)\mathscr{S}_+\left(\frac{2\pi}{q},\theta\right)\:+\: \mathcal{O}(q)\:,
\end{eqnarray}
where
\begin{equation*}
\mathscr{S}_+\left(\frac{2\pi}{q},\theta\right) = \sum_{n=0}^{q-1} e^{-2\pi i \frac{n(n+\sigma+\tau+1)}{q}} \:e^{in\theta}\:.
\end{equation*}
Next, if $\theta=\frac{2\pi p}{q}$ for some integer $p$, then we get
\begin{equation*}
\mathscr{S}_+\left(\frac{2\pi}{q},\frac{2\pi p}{q}\right) = \sum_{n=0}^{q-1} e^{-2\pi i \frac{n(n+\sigma+\tau+1)}{q}} \:e^{2\pi i\frac{np}{q}}=\sum_{n=0}^{q-1} e^{-2\pi i \frac{n^2+n(\sigma+\tau+1-p)}{q}} \:.
\end{equation*}
Now completing the square in the numerator of the last expression modulo $q$ and plugging in the evaluation of the complete Gauss sum, we get that
\begin{equation*}
\mathscr{S}_+\left(\frac{2\pi}{q},\frac{2\pi p}{q}\right) = z_q\: \sqrt{q}\: e^{2\pi i\frac{r(\sigma+\tau+1-p)^2}{q}}\:,
\end{equation*}
where $r$ is an integer such that $4r \equiv 1 (\text{mod }q)$ and $z_q$ is a unimodular complex number as $q$ is odd. Now choosing $p$ to be even, we can utilize the property that $4r \equiv 1 (\text{mod }q)$ to get,
\begin{eqnarray} \label{result3_pf_eq4}
\mathscr{S}_+\left(\frac{2\pi}{q},\frac{2\pi p}{q}\right) &=& z_q\: \sqrt{q}\: e^{2\pi i\frac{r(\sigma+\tau+1)^2 + rp^2}{q}}\:e^{-2\pi i \frac{2r(\sigma + \tau +1)p}{q}} \nonumber\\
&=& z_q\: \sqrt{q}\: e^{2\pi i\frac{r(\sigma+\tau+1)^2 + rp^2}{q}}\:e^{- i \frac{\pi(\sigma + \tau +1)p}{q}}\:.
\end{eqnarray}
Next, we consider $\theta=\frac{2\pi p}{q} + \xi$, where $p,q$ are as before and $0< \xi < \pi/(8N)$. Then since the chordal metric on $S^1$ is $1$-Lipschitz, it follows that
\begin{equation*}
\left|\mathscr{S}_+\left(\frac{2\pi}{q},\theta\right)-\mathscr{S}_+\left(\frac{2\pi}{q},\frac{2\pi p}{q}\right)\right| \le \sum_{n=0}^{q-1}|e^{in\xi}-1| \le \xi \left(\sum_{n=0}^{q-1} n \right) \le \frac{q^2}{N}\:,
\end{equation*}
that is,
\begin{equation} \label{result3_pf_eq5}
\mathscr{S}_+\left(\frac{2\pi}{q},\theta\right) = \mathscr{S}_+\left(\frac{2\pi}{q},\frac{2\pi p}{q}\right) + \mathcal{O}\left(\frac{q^2}{N}\right)\:.
\end{equation}
Also,
\begin{equation*}
\left|e^{i\kappa_{\sigma,\tau}(\theta)}-e^{i\kappa_{\sigma,\tau}\left(\frac{2\pi p}{q}\right)}\right|=\left|e^{\frac{i}{2}(\sigma + \tau +1)\xi}-1\right| \lesssim \frac{1}{N}\:,
\end{equation*}
that is,
\begin{equation} \label{result3_pf_eq6}
e^{i\kappa_{\sigma,\tau}(\theta)} = e^{i\kappa_{\sigma,\tau}\left(\frac{2\pi p}{q}\right)} + \mathcal{O}\left(\frac{1}{N}\right)\:.
\end{equation}
Now combining (\ref{result3_pf_eq3})-(\ref{result3_pf_eq6}), we get
\begin{equation*}
S_{+}\left(\frac{2\pi}{q},\theta\right)\:e^{i\kappa_{\sigma,\tau}(\theta)} = \left(\sum_{\ell=0}^{\left\lfloor\frac{N}{q}\right\rfloor - 1} e^{i\ell q \theta}\right)\mathscr{S}_+\left(\frac{2\pi}{q},\frac{2\pi p}{q}\right)e^{i\kappa_{\sigma,\tau}\left(\frac{2\pi p}{q}\right)} + \mathcal{O}(q)\:.
\end{equation*}
Next plugging in the explicit expression (\ref{result3_pf_eq4}) in the above, we further get,
\begin{equation*}
S_{+}\left(\frac{2\pi}{q},\theta\right)\:e^{i\kappa_{\sigma,\tau}(\theta)} =z_q\: \sqrt{q}\: e^{2\pi i\frac{r(\sigma+\tau+1)^2 + rp^2}{q}} \left(\sum_{\ell=0}^{\left\lfloor\frac{N}{q}\right\rfloor - 1} e^{i\ell q \xi}\right)e^{-i\frac{\left(\sigma + \frac{1}{2}\right)}{2}\pi} + \mathcal{O}(q)\:.
\end{equation*}
Similarly, for the same values of $\theta$, we will get
\begin{equation*}
S_{-}\left(\frac{2\pi}{q},\theta\right)\:e^{-i\kappa_{\sigma,\tau}(\theta)} =z_q\: \sqrt{q}\: e^{2\pi i\frac{r(\sigma+\tau+1)^2 + rp^2}{q}} \left(\sum_{\ell=0}^{\left\lfloor\frac{N}{q}\right\rfloor - 1} e^{-i\ell q \xi}\right)e^{i\frac{\left(\sigma + \frac{1}{2}\right)}{2}\pi} + \mathcal{O}(q)\:.
\end{equation*}
From the above two estimates, we get back to $S\left(\frac{2\pi}{q},\theta \right)$ in (\ref{St}),
\begin{equation*}
\left|S\left(\frac{2\pi}{q},\theta \right)\right| = \sqrt{q}\left|\sum_{\ell=0}^{\left\lfloor\frac{N}{q}\right\rfloor - 1} \cos\left(\frac{\left(\sigma + \frac{1}{2}\right)}{2}\pi - \ell q \xi\right)\right| \:+\:\mathcal{O}\left(\sqrt{N}\right)\:.
\end{equation*}
At this juncture, we note that $\frac{\left(\sigma + \frac{1}{2}\right)}{2}=\frac{d-1}{4}$. Thus for $\X=P^d(\C),\:P^d(\Hb)$ and $P^{16}(Cay)$, as $d$ is even and $\ell q \xi < \pi/8$, the cosines in the sum are of the same sign and moreover, we get that there exists $C>0$ such that each of the summands satisfy,
\begin{equation*}
\left|\cos\left(\left(\frac{d-1}{4}\right)\pi - \ell q \xi\right)\right| \ge C\:.
\end{equation*}
And hence,
\begin{equation*}
\left|S\left(\frac{2\pi}{q},\theta \right)\right| \gtrsim \sqrt{q} \times \frac{N}{q} -\sqrt{N} \gtrsim N^{3/4}\:.
\end{equation*}
On the other hand, is $\X=P^d(\R)$, $d$ odd, then again as $\ell q \xi < \pi/8$, the cosines in the sum are of the same sign and satisfy,
\begin{equation*}
\left|\cos\left(\left(\frac{d-1}{4}\right)\pi - \ell q \xi\right)\right| \ge C \begin{cases}
	 \ell q \xi  &\text{ if }  d \equiv 3 \text{ or } 7 \text{ (mod }8)\:, \\
	1\: &\text{ if }  d \equiv 1 \text{ or } 5 \text{ (mod }8)\:.
	\end{cases}
\end{equation*}
Thus for $\xi \in \left(\frac{\pi}{16N},\:\frac{\pi}{8N}\right)$, we again get
\begin{equation*}
\left|S\left(\frac{2\pi}{q},\theta \right)\right| \gtrsim  N^{3/4}\:.
\end{equation*}
Hence, in either case, plugging the above estimate in (\ref{result3_pf_eq2}), we get that in particular for $\theta \in E_N$,
\begin{equation*} 
\sup_{0\le t <2\pi}\left| e^{it\Delta} f_N(\theta)\right| \gtrsim N^{3/4}\:,
\end{equation*}
that is, we get (\ref{result3_pf_eq1}), which completes the proof for the Schr\"odinger equation.

\medskip

{\bf Step 2:} By the transference principle (Theorem \ref{transference_principle}), the results in Theorem \ref{result3} for the Boussinesq equation and the Beam equation follow from the result of the Schr\"odinger equation proved in Step $1$, in view of Remark \ref{examples_remark}. 

Indeed, for the sake of contradiction, let us assume that there exists $\alpha_0<1/4$ for which the maximal estimate holds for Beam or Boussinesq. Then the maximal estimate also holds for all $\alpha > \alpha_0$. By the transference principle (Theorem \ref{transference_principle}) and Remark \ref{examples_remark}, the same must also be true for the Schr\"odinger equation, in particular, for all $\alpha \in (\alpha_0,1/4)$ which contradicts the result obtained in Step 1. This completes the proof of Theorem \ref{result3}.
\end{proof}

\begin{remark} \label{remark_failure_result3}
Our number theoretic arguments crucially use the fact that $\sigma + \tau +1 \in \N$, so that one can use the $q$-periodicity to obtain the decomposition (\ref{result3_pf_eq3}). Clearly, this fails for $P^d(\R)$ with $d$ even.  
\end{remark}

\section{Open problems}
It would be interesting to pursue a complete solution to the Carleson's problem on Riemannian symmetric spaces of compact type. It seems to be a difficult task, however, given that it has not been achieved yet, even in the case of the circle $\T$. In this light, we pose the following problems:
\begin{enumerate}
\item What are possible analogues of Theorem \ref{result1}, for rank $\ge 3$?
\item Can Theorem \ref{result3}, that is, the failure of $\alpha <1/4$ be extended to $P^d(\R)$, for $d$ even?
\item Finally, can the results appearing in this article be improved?
\end{enumerate}

\section*{Acknowledgements} We are thankful to Giacomo Gigante for several illuminating comments and suggestions. We also thank Kummari Mallesham for discussions. The first author is supported by the Institute Post Doctoral Fellowship of Indian Institute of Technology, Bombay.

\bibliographystyle{amsplain}

\begin{thebibliography}{amsplain}
\bibitem[Bou92]{Bourgain1} Bourgain, J. {\em A remark on Schr\"odinger operators.} Israel J. Math. 77 (1992), 1-16.
\bibitem[Bou93]{Bourgain-GAFA} Bourgain, J. {\em Fourier transform restriction phenomena for certain lattice subsets and applications to nonlinear evolution equations.} Geom. Funct. Anal. 3 (1993), 107-156.
\bibitem[Bou16]{Bourgain} Bourgain, J. {\em A note on the Schr\"odinger maximal function.} J. Anal. Math. 130, 393-396 (2016).
\bibitem[BD15]{BD} Bourgain, J. and Demeter, C. {\em The proof of the $\ell^2$ decoupling conjecture.} Ann. Math. 182, no. 1 (2015): 351–89.
\bibitem[BGG25]{Gigante} Brandolini, L., Gariboldi, B. and Gigante, G. {\em Irregularities of distribution on two-point homogeneous spaces.} In: Hern\'andez, E., Peloso, M.M., Ricci, F., Soria, F., Tabacco, A. (eds) The Mathematical Heritage of Guido Weiss. Applied and Numerical Harmonic Analysis. Birkh\"auser, Cham (2025).
\bibitem[Car80]{C} Carleson, L. {\em Some analytic problems related to statistical mechanics, Euclidean harmonic analysis.} Lecture Notes in Math. 779, Springer, Berlin, 5-45 (1980).
\bibitem[CDLY22]{CDLY} Chen, X., Duong, X.T., Lee, S. and Yan, L. {\em A sharp regularity estimate for the Schr\"odinger propagator on the sphere.} J. Math. Pures Appl. 163 (2022), 433-449.
\bibitem[CT17]{CT} Colzani, L. and Tenconi, M. {\em 
Localization for Riesz means on the compact rank one symmetric spaces.} Proceedings of the AMSI/AustMS 2014 Workshop in Harmonic Analysis and its Applications, Austral. Nat. Univ., (47) 26–49,  2017.
\bibitem[CLS21]{CLS} Compaan, E., Luc\`a, R. and Staffilani, G. {\em Pointwise Convergence of the Schr\"odinger Flow.}
Int. Math. Res. Not., Vol. 2021, No. 1, pp. 599–650
\bibitem[Cow83]{Cowling} Cowling, M. {\em Pointwise behavior of solutions to Schr\"odinger equations, harmonic analysis.} Lecture Notes in Math. 992. Springer, Berlin, 83-90 (1983).
\bibitem[DK82]{DK} Dahlberg, B.E.J. and Kenig, C.E. {\em A note on the almost everywhere behavior of solutions of the Schr\"odinger equation.} Lecture Notes in Math. 908. Springer-Verlag, Berlin, 205-208 (1982).
\bibitem[DGL17]{DGL} Du, X., Guth, L. and Li, X. {\em A sharp Schr\"odinger maximal estimate in $\R^2$.} Ann. Math. 186, 607-640 (2017).
\bibitem[DZ19]{DZ} Du, X. and Zhang, R. {\em Sharp $L^2$ estimates of the Schr\"odinger maximal function in higher dimensions.} Ann. Math. 189, 837-861 (2019).
\bibitem[FG96]{FG} Fang, Y-F. and Grillakis, M. G. {\em Existence and uniqueness for Boussinesq type equations on a circle.} Communications in Partial Differential Equations, 21(7-8), 1253-1277 (1996).
\bibitem[GPW08]{GPW} Guo, Z., Peng, L. and Wang, B. {\em Decay estimates for a class of wave equations.} J. Funct. Anal. 254 (2008), 1642-1660.
\bibitem[HW75]{Hardy} Hardy, G.H. and Wright, E. M. {\em An introduction to the theory of numbers.} Fourth edition, Oxford University Press, 1975.
\bibitem[Hel78]{HelgasonDiff} Helgason, S. {\em Differential geometry, Lie groups and Symmetric spaces.} Pure and Applied Mathematics, Academic Press, Inc., New York-London, 1978.
\bibitem[Hel08]{HelSymm} Helgason, S. {\em Geometric Analysis on Symmetric Spaces.} 2nd ed. Mathematical Surveys and Monographs, vol. 39. Providence, RI: American Mathematical Society, 2008.
\bibitem[Hel00]{Helgason} Helgason, S. {\em Groups and geometric analysis, Integral geometry, invariant differential operators, and spherical functions}. Mathematical Surveys and Monographs, vol. 83. Providence, RI: American Mathematical Society, (2000).
\bibitem[Lee06]{Lee} Lee, S. {\em On pointwise convergence of the solutions to Schr\"odinger equations in $\R^2$.} Int. Math. Res. Not. Art. ID 32597, 1-21 (2006).
\bibitem[MV08]{MV} Moyua, A. and L. Vega. {\em Bounds for the maximal function associated to periodic solutions of one-dimensional dispersive equations.} Bull. Lond. Math. Soc. 40, no. 1 (2008): 117–128.
\bibitem[MVV96]{MVV} Moyua, A., Vargas, A. and Vega, L. {\em Schr\"odinger maximal function and restriction properties of the Fourier transform.} Int. Math. Res. Not. 16(1996), 793-815.
\bibitem[\'OP13]{Pasquale} \'Olaffson, G. and Pasquale, A. {\em Ramanujan’s Master Theorem for the Hypergeometric
Fourier Transform Associated with Root Systems.} J Fourier Anal Appl (2013) 19:1150–1183\:.
\bibitem[\'OS08]{OS} \'Olaffson, G. and Schlichtkrull, H. {\em A local Paley–Wiener theorem for compact symmetric spaces.} Advances in Mathematics 218 (2008) 202–215.
\bibitem[Pro07]{Pro} Procesi, C. {\em Lie groups. An approach through invariants and representations.} Universitext. Springer, New York, 2007. 
\bibitem[Sj\"o87]{Sjolin} Sj\"olin, P. {\em Regularity of solutions to the Schr\"odinger equation.} Duke Math J. 55, 699-715 (1987).
\bibitem[Sze75]{Szego} Szeg\"o, G. {\em Orthogonal Polynomials.} Fourth edition. American Mathematical Soc., Providence, Rhode Island, 1975.
\bibitem[TV00]{Tao} Tao, T. and Vargas, A. {\em A bilinear approach to cone multipliers. II. Applications.} Geom. Funct. Anal. 10(2000), 216-258.
\bibitem[Veg88]{Vega} Vega, L. {\em Schr\"odinger equations: pointwise convergence to the initial data.} Proc. Amer. Math. Soc. 102, 874-878 (1988).
\bibitem[Wan52]{Wang} Wang, H.-C. {\em Two-point homogeneous spaces.} Ann. Math. (2), 55 (1952), pp. 177-191.
\bibitem[WZ19]{WZ} Wang, X. and Zhang, C. {\em Pointwise Convergence of Solutions to the Schr\"odinger Equation on Manifolds}. Canad. J. Math. Vol. 71(4), 983-995 (2019).


\end{thebibliography}

\end{document}